\documentclass[11pt, final]{amsart}
\usepackage{amsmath,amsthm,amssymb,amsfonts}
\usepackage{xcolor}
\usepackage[colorlinks=false, final]{hyperref}
\usepackage{verbatim}
\usepackage{enumitem}
\usepackage{xspace}
\usepackage{geometry}
\usepackage{pst-node}
\usepackage{tikz-cd}
\setlist[enumerate,1]{label=(\arabic*)} 
\theoremstyle{plain}
\newtheorem{thm}{Theorem}[section]
\newtheorem{theorem}[thm]{Theorem} 
\newtheorem{proposition}[thm]{Proposition}
\newtheorem{lemma}[thm]{Lemma}
\newtheorem{corollary}[thm]{Corollary}
\theoremstyle{definition}
\newtheorem{nota}[thm]{Notation}
\newtheorem{definition}[thm]{Definition}
\theoremstyle{remark}
\newtheorem{remark}[thm]{Remark}
\newtheorem{example}[thm]{Example}
\newtheorem{question}[thm]{Question}
\counterwithin{equation}{section}
\counterwithin{figure}{section}

\makeatletter
\providecommand\@dotsep{5}
\makeatother
\title{Subset expansions of monoids}
\date{\today}
\author[V. Gould]{Victoria Gould}
\address{University of York}
\email{victoria.gould@york.ac.uk}
\author[M. Johnson]{Marianne Johnson}
\address{University of Manchester}
\email{Marianne.Johnson@manchester.ac.uk}
\date{\today}


\begin{document}
\begin{abstract} We initiate the study  of the expansion  $\mathcal{S}(M)$ of  a monoid $M$ obtained via the semidirect product  of $M$ acting naturally on the left of its power set (regarded as a semilattice under union).  We term this the `subset expansion' of $M$. The  monoid $\mathcal{S}(M)$ contains the images of several expansions of $M$ of wide interest and use in semigroup theory, in particular the prefix and Szendrei  expansions (in the case where $M$ is free, these `smaller' expansions produce free algebras in certain varieties).

We first focus on algebraic properties, specifically those determined by idempotents. We note that $\mathcal{S}(M)$ is regular if and only if $M$ is a group and determine the conditions on $M$ that ensure   $\mathcal{S}(M)$ possesses certain generalised regularity properties. Particularly, we show that the expansion $\mathcal{S}$ maps groups to proper inverse monoids, unipotent monoids to proper left restriction monoids, right cancellative monoids to left ample monoids, right abundant monoids to right abundant monoids, and left cancellative monoids to right adequate monoids.

Subsequently, we focus on  finitary conditions. We examine the condition of weak left  coherence (every finitely generated left ideal has a finite presentation as a left act); the related conditions of property (L), left ideal Howson, finitely left equated, and each of the corresponding left-right dual notions. Each of these conditions is preserved under retract, from which it is immediate that if $\mathcal{S}(M)$ satisfies one of our finitary conditions, then so must $M$, but the converse is not true. For each of our properties to `lift' from $M$ to $\mathcal{S}(M)$ it must undergo a `strengthening'. Indeed, we show that $\mathcal{S}(M)$ satisfies property (L) (or its left-right dual) if and only if $M$ is finite. We provide exact characterisations of the monoids $M$ such that $\mathcal{S}(M)$ is: left (or right) ideal Howson; finitely left equated; and (consequently) weakly left coherent. We give sufficient conditions for $\mathcal{S}(M)$ to be finitely right equated and hence weakly right coherent. 
\end{abstract}
\maketitle
\section{Introduction}
\label{sec:prod}
Monoid expansions provide systematic ways to build a new monoid  from a given input monoid, often by inflating the set of idempotents in some structured manner. They were introduced in \cite{BR84} to facilitate the application of techniques in finite semigroup theory to infinite monoids and semigroups. Most well-studied examples of expansions have the property that when restricting the input to an interesting class of monoids (e.g. groups, cancellative monoids, unipotent monoids), the output of the expansion is a monoid with desirable structural properties.  For example, the Szendrei expansion \cite{Sz89} maps: groups to inverse monoids \cite{BR84}, right cancellative monoids to left ample monoids \cite{FG90} and unipotent monoids to (what are now called) left restriction monoids \cite{FGG99}. In fact, \cite{BR84} used an earlier formulation, known as the prefix expansion, which is equivalent to the Szendrei expansion in the case the monoid is a group; \cite{Ho07} investigates the prefix expansion of an arbitrary monoid.  The underlying construction used in these expansions is a semidirect product formed from the monoid $M$ acting on a particular semilattice of  subsets of $M$. It is then natural to consider the `parent' expansion where the semilattice contains {\em all} subsets of $M$, and this is our focus here. Specifically,  this paper considers the expansion formed via the action of a monoid $M$  on its power set $\mathcal{P}(M)$ viewed as a monoid (indeed, a semilattice) with respect to the operation of union, and with identity element $\emptyset$, with  left action of $M$ on $\mathcal{P}(M)$ given by $( m,X) \mapsto mX = \{mx: x \in X\}$ for all $X \in \mathcal{P}(M)$ and all $m \in M$. Throughout the paper we will write $\mathcal{S}(M)$ to denote $\mathcal{P}(M) \rtimes M$ with product given by $(X,x)(Y,y) = (X \cup xY, xy)$, and with identity element $(\emptyset, 1)$, where $1$ is the identity element of $M$. We  refer to monoids of this form as `subset expansions' to indicate that this construction involves the action of $M$ on its subsets.

Our interest in subset expansions comes from several directions. First, $\mathcal{S}(M)$ contains the Szendrei expansion of $M$ as a monoid subsemigroup, and (as with the Szendrei expansion) if $G$ is a group then $\mathcal{S}(G)$ is an inverse monoid. If $M=F_X$  is the free group on a
set $X$, we remark that the monoid $\mathcal{S}(F_X)$ contains: the free inverse, free ample and free
left ample monoids on the same set $X$. For a monoid $M$ which is not a group, the semidirect product $\mathcal{S}(M)$ is no longer inverse (indeed, it is not even regular; see Lemma \ref{lem:Sreg} below for details). It is then natural to ask under which circumstances $\mathcal{S}(M)$  belongs to one of several more general classes of monoids (e.g. left/right restriction, ample, adequate, abundant).

Second, we note that any inverse monoid satisfies several \emph{finitary conditions}  (that is, conditions satisfied by every finite monoid), including weak left (respectively, right) coherence \cite{BGR23}. A monoid $M$ is said to be weakly left (respectively, right) coherent if every finitely generated left (respectively, right) ideal of $M$ is finitely presented as a left (respectively, right) $M$-act. A classical result of Wheeler \cite{Wh76}, interpreted for acts in \cite{Go87}, says that a monoid is left coherent if and only if the class of existentially closed left acts is axiomatisable. Inspired by this, weak left coherency for monoids, and the related finitary conditions we consider here, have also  been formulated in terms of the model-theoretic notion of axiomatisability: see \cite{Go871,Sh12,DGM24}.  In previous work \cite{GJ25} the construction $\mathcal{S}(M)$ has been utilised to exhibit some interesting examples of monoids that are neither left nor right coherent.  We have remarked that  the monoid $\mathcal{S}(F_X)$ contains the free inverse, free ample and free left ample monoids on set $X$, and it is known that each of the latter is weakly left and right coherent \cite{BGR23,CG21,GH}; cf. the table in \cite{GJ25}. One is then drawn to ask (see Question 8.6 of \cite{GJ25}): for which monoids $M$ is $\mathcal{S}(M)$ weakly left (respectively right) coherent?  It is known that $M$ is weakly left coherent if and only if it is simultaneously left ideal Howson (the
intersection of any two finitely generated left ideals is finitely generated) and finitely left equated (the left
annihilator congruences of $M$ are finitely generated as left congruences). We initiate the study of finitary conditions in the context of expansions of monoids by examining the behaviour of these two conditions and the closely related conditions of (R) and (L) with respect to the expansion $\mathcal{S}$.

The paper is structured as follows. In Section 2 we outline the preliminary definitions and results required for our paper and  in Section 3 we prove some structural results on subset expansions. Our main results are in Section 4. We give precise characterisations of the monoids $M$ for which (R), (L) (Subsection~\ref{sub:RL}) or the left and right ideal Howson properties hold in $\mathcal{S}(M)$ (Subsection~\ref{sub:howson}). We characterise the monoids $M$ for which $\mathcal{S}(M)$ is finitely left equated. We exhibit a configuration of elements in $M$ that prevent $\mathcal{S}(M)$ from being finitely right equated.  We introduce a notion of being strongly finitely right equated and show that $\mathcal{S}(M)$ satisfies this condition if and only if $M$ does. (Subsection~\ref{sub:fre}).  In particular, we apply our results to demonstrate that if $M$ is a right (left) LCM monoid then $\mathcal{S}(M)$ is weakly right (left) coherent. We provide a number of examples and counterexamples to illustrate and separate the properties under consideration.  We finish with some open questions.

\section{Preliminaries}
In this section we outline the  background necessary for this paper. Where the notions we discuss have clear left-right duals we do not mention them explicitly, except to set up notation. We refer the reader to \cite{H} for a fuller account of the semigroup theoretic notions we discuss. Throughout the paper $M$ denotes a monoid with identity element $1$. 

\subsection{Right congruences} An equivalence relation $\rho$ on $M$ is a {\em right congruence} if for all $a,b,c\in M$ if $a\,\rho\, b$ then $ac\,\rho\, bc$. Every right ideal of $M$ determines a Rees right congruence, but not all right congruences correspond to right ideals. If $W\subseteq M\times M$ then we denote the smallest right congruence containing  $W$ by $\rho_W$ and refer to this as the right congruence {\em generated} by $W$. The following result is standard and may be found in \cite{KKM}.

\begin{proposition}\label{prop:gen} Let $W\subseteq M\times M$ and let $a,b\in M$. Then $a\,\rho_W\, b$ if and only if there is a sequence
\[a=c_1t_1,\, d_1t_ 1=c_2t_2,\, \cdots, \, d_nt_n=b,\]
where $n\geq 0$,  $(c_i,d_i)\in W\cup W^{-1}$ and $t_i\in M$, for $1\leq i\leq n$.
\end{proposition}
A sequence as in Proposition~\ref{prop:gen} is called a {\em $W$-sequence of length  $n$}; we interpret $n=0$ as saying that $a=b$. The dual result holds for left congruences; we denote the left congruence generated by $W$ by $\lambda_W$.

\subsection{Classes of monoids determined by extensions of Green's relation $\mathcal{L}$}
We recall several well-known classes of semigroups that may be viewed as generalisations of regular (every $\mathcal{L}$-class and every $\mathcal{R}$-class contains an idempotent) and inverse (regular and the idempotents form a semilattice) monoids determined by relations that extend Green's relations $\mathcal{L}$ and $\mathcal{R}$ with respect to a distinguished semilattice of idempotents. We recap the relevant definitions extending $\mathcal{L}$; for brevity we omit the natural left-right dual definitions extending $\mathcal{R}$. 

\begin{definition}(Right annihilator congruence, $\mathcal{L}^*$, $\tilde{\mathcal{L}}_E$, right abundant, right Fountain)\label{defn:rac}  For $a \in M$ let
$$\mathbf{r}(a) := \{(u, v) \in  M \times M : au 
\, = \, av\}.$$
It is straightforward to check that $\mathbf{r}(a)$ is a right congruence; we call this the \emph{right annihilator congruence} of $a$. We denote by $\mathcal{L}^*$ the equivalence relation on $M$ defined (for all $a,b \in M$) by $a \,\mathcal{L}^*\, b$ if $\mathbf{r}(a) = \mathbf{r}(b)$. The relation $\mathcal{L}^*$ is a generalisation of  Green's $\mathcal{L}$-relation in that it contains   $\mathcal{L}$ and is a right congruence.  Let $E(M)$ denote the set of
idempotents of $M$, and let $E \subseteq E(M)$. The pre-order  $\leq_{\widetilde{\mathcal{L}}_E}$ on $M$ is defined by the rule that for all $a,b\in M$:
\[a\leq_{\widetilde{\mathcal{L}}_E} b\mbox{ if and only if } \{ e\in E: be=b\} \subseteq \{ e\in E: ae=a\}.\] The associated equivalence relation is denoted by $\widetilde{\mathcal{L}}_E$.
Thus, $a\,\widetilde{\mathcal{L}}_E\, b$
if and only if $a$ and $b$ have the same  right identities from $E$. 
The relation $\widetilde{\mathcal{L}}_E$ is a generalisation of both Green's $\mathcal{L}$-relation 
and the relation $\mathcal{L}^*$: indeed, we have that $\mathcal{L} \subseteq \mathcal{L}^* \subseteq \mathcal{L}_E$, and if $
M$ is regular we have $\mathcal{L} = \widetilde{\mathcal{L}}_{E(M)}$. In general, however, unlike the relations $\mathcal{L}$ and $\mathcal{L}^*$, the relation $\widetilde{\mathcal{L}}_E$ need not be a right congruence.  We say that $M$ is \emph{right abundant} (respectively, \emph{right Fountain}) if every $\mathcal{L}^*$--class (respectively, $\widetilde{\mathcal{L}}$--class) of $M$ contains an idempotent. 
\end{definition}

\begin{definition}(Right $E$-Ehresmann, $E$-adequate, $E$-ample and $E$-restriction monoids)
Suppose now that $E \subseteq E(M)$ is a semilattice (for example, one could take $E=\{1\}$). Then each $\widetilde{\mathcal{L}}_E$-class of $M$ contains at most one idempotent of $E$, and the monoid $M$ is said to be {\em right $E$-Ehresmann} if every $\widetilde{\mathcal{L}}_E$-class of $M$ contains a (unique) element of $E$ and $\widetilde{\mathcal{L}}_E$ is a right congruence.  In a right $E$-Ehresmann monoid then, there is a natural unary operation determined by $x \mapsto x^*$ for all $x \in M$ where $x^*$ denotes the unique element of $E$ in the $\widetilde{\mathcal{L}}_E$ class of $x$. If $M$ is right $E$-Ehresmann and $\widetilde{\mathcal{L}}_E = \mathcal{L}^*$, then $M$ is said to be \emph{right $E$-adequate}.
We say that $M$ is \emph{right $E$-ample} (respectively, \emph{right $E$-restriction}) if it is right $E$-adequate (respectively right $E$-Ehresmann) and satisfies the right ample identity  $ea = a(ea)^*$ for all $a \in M$ and all $e \in E$.  We denote by $\sigma_E$ the least (two-sided) congruence on $M$ identifying all elements of $E$ or simply $\sigma$ in the case where $E=E(M)$. An inverse  (respectively, right $E$-ample; right $E$-restriction) monoid $M$ is said to be \emph{proper} if $\mathcal{L} \, \cap\,  \sigma$ (respectively, $\mathcal{L}^* \, \cap \sigma_E$; $\tilde{\mathcal{L}}_E \, \cap \sigma_E$) is trivial.
\end{definition}

\begin{remark}
In a left $E$-Ehresmann monoid we use the notation $x^+$ to denote the unique element of $E$ in the $\widetilde{\mathcal{R}}$-class of $x$. The left ample identity is then $ae = (ae)^+a$ for all $a \in M$ and all $e \in E$.
\end{remark}

\begin{remark}
If $E(M)$ is itself a semilattice, then we may take $E=E(M)$ in each of the above definitions; in this case we suppress the $E$ denoting the corresponding relation by $\tilde{\mathcal{L}}$ and using the terminology right Ehresmann, right adequate, right ample and right restriction. For consistency we keep this terminology although we note that if $M$ is right $E$-ample then it turns out that we must in fact have $E=E(M)$. To see this, let $e\in E(M)$ and notice that since $e\, \widetilde{\mathcal{L}}_E \, e^*$ we have  $ee^* = e$ giving $ee=ee^*$ so that (using that $e\, \mathcal{L}^*_E \, e^*$) $e^*e=e^*e^*=e^*$; applying the right ample identity then gives $e^* =e^*e = e(e^*e)^* = e(e^*)^* = ee^* = e$ and hence $e \in E$.
\end{remark}

\subsection{A suite of finitary conditions}
Weak right coherence is a finitary condition for monoids (in the sense that every finite monoid is weakly coherent) and is defined as follows:
\begin{definition}(Weak right coherence)
\label{sec:weak}
A monoid $M$ is \emph{weakly right coherent} if every finitely generated right ideal of M is finitely presented as a right $M$-act. 
\end{definition}
In practice, we shall make use of the following characterisation of weak right coherence.
\begin{theorem}
\cite[Corollary 3.3]{G}
\label{thm:weakcoherent}
A monoid $M$ is weakly right coherent if and only if
\begin{enumerate}    \item $aM \cap bM$ is finitely generated;
\item 
$\mathbf{r}(a)$ is a finitely generated right congruence on $M$.
\end{enumerate}
\end{theorem}

It will be convenient to have some terminology corresponding to the conditions of the previous theorem.
\begin{definition}(Right ideal Howson and finitely right equated monoids) A monoid with the property that the intersection of any two principal (or equivalently two finitely generated) right ideals is finitely generated is said to be {\em right ideal Howson}. A monoid $M$ with the property that each right annihilator congruence of the form $\mathbf{r}(a)$ where $a \in M$ is finitely generated are said to be {\em finitely right equated}. 
\end{definition}

Another related finitary condition arises as follows. Here the {\em right diagonal act} of $M$ is the set $M \times M$,
with co-ordinatewise right action by $M$, that is, $(a,b)m=(am,bm)$ for all $(a,b)\in M\times M$ and $m\in M$. A  right subact $A$ of $M\times M$ is finitely generated if there is a finite set $X\subseteq A$ such that
$A=XM$.

\begin{definition}(Condition (R))\label{defn:R}
For $a,b \in M$ let
$$\mathbf{R}(a,b) := \{(u, v) \in  M \times M : au \, = \, bv\} \subseteq M \times M.$$
It is easy to see that $\mathbf{R}(a,b)$ is a right subact of the right diagonal act of $M$. We say that $M$ {\em satisfies condition (R)} if for all $a,b \in M$ the right subact $\mathbf{R}(a,b)$ is finitely generated. \end{definition}

Each of these properties behaves well with respect to retraction:

\begin{theorem}\cite[Corollary 3.5, Corollary 4.12, Theorem 5.5]{DGM24},  \cite[Theorem 6.2]{GH} and \cite[Theorem 2.3]{GHS}
\label{thm:retract}
Let $M$ be a monoid and $S$ be a retract of $M$. If $M$ is right ideal Howson (respectively, finitely right equated; weakly right coherent, satisfies condition (R)), then so is $S$.
\end{theorem}

We note that the term `finitely right aligned' used elsewhere in the literature coincides with the property of being right ideal Howson for \emph{monoids}. (Right ideal Howson \emph{semigroups} need not be finitely aligned, however \cite{CG21}.)

For right abundant monoids, the notions of weak right coherence and right ideal Howson coincide:\begin{proposition}
\label{prop:rabundweak}
\cite[Lemma 3.5]{G}
Every right abundant monoid is finitely right  equated.
\end{proposition}

\section{Properties of subset expansions}

We recall briefly that an expansion of monoids is a functor $\mathcal{F}$ from the category of monoids to the category of monoids such that there is a natural transformation $\rho$ from $\mathcal{F}$ to the identity functor, with surjective components, that is: for any two monoids $M, N$ there exist surjective morphisms $\rho_M: \mathcal{F}(M) \rightarrow M$ and $\rho_N: \mathcal{F}(N) \rightarrow N$ such that for any morphism $\varphi: M \rightarrow N$, the following diagram commutes:
\[ \begin{tikzcd}
\mathcal{F}(M) \arrow{r}{\mathcal{F}(\varphi)} \arrow[swap]{d}{\rho_M} & \mathcal{F}(N)\arrow{d}{\rho_N} \\
M\arrow{r}{\varphi}& N
\end{tikzcd}
\]
As mentioned in the introduction $\mathcal{S}: M \rightarrow \mathcal{S}(M)$ is an expansion of monoids, as is easily seen by considering the surjection $(A,a) \mapsto a$ onto the second co-ordinate. We now record some basic facts about the structure of $\mathcal{S}(M) := \mathcal{P}(M) \rtimes M$

The following notational conventions will be useful.

\begin{nota}
For $a \in M$ and $A \subseteq M$ we write $a^{-1}A$ and $\hat{A}$ to denote the sets 
\begin{eqnarray*}
a^{-1}A&:=&\{u \in M: au \in A\}\\
\hat{A}&:=&\{e \in E(M): eA \subseteq A\}.
\end{eqnarray*}
\end{nota}

\begin{remark}
\label{rem:inclusions}
For $a \in M$ and $X,Y \subseteq M$, $aY \subseteq X$ if and only if $Y \subseteq a^{-1}X$. In particular for any $X \subseteq M$ we have:
\begin{eqnarray}
\label{eq:a-a}X \subseteq a^{-1}(aX) &=& \{u: u \in M, au  \in aX\},\\
\label{eq:aa-}a(a^{-1}X) &=& \{au: u \in M, au \in X\}  = aM \cap X \subseteq X.
\end{eqnarray}

\begin{enumerate}
\item It follows from \eqref{eq:a-a} that $a^{-1}(aX) = X$ for all $a \in M$ and all $X \subseteq M$ if and only if $M$ is left cancellative. (In general the inclusion $\{b\} \subseteq a^{-1}(\{ab\})$ can be strict.)
\item It follows from \eqref{eq:aa-} that $a(a^{-1}X) = X$ for all $a \in M$ and all $X \subseteq M$ if and only if $M$ is a group. (In general the inclusion $a(a^{-1}M) = aM \subseteq M$ can be strict.)
\end{enumerate}
\end{remark}

\begin{lemma}
\label{lem:Sreg}
Let $M$ be a monoid.
\begin{enumerate}
\item The set of idempotents of $\mathcal{S}(M)$ is $\{(A,e): e \in \hat{A}\}$ and this is a semilattice precisely if $1$ is the unique idempotent of $M$. In general, $\mathcal{E}:= \{(A,1): A \subseteq M\}$ is a semilattice of idempotents.
\item The monoid $\mathcal{S}(M)$ is regular if and only if $M$ is a group, in which case $\mathcal{S}(M)$ is a proper inverse monoid (and hence, in particular, weakly left and right coherent).
\end{enumerate}
\end{lemma}

\begin{proof}
(1) The idempotents of $\mathcal{S}(M)$ are easily seen to be the elements of the form $(A,e)$ where $e$ is an idempotent of $M$ and $A \subseteq M$ satisfies $eA \subseteq A$. In particular $(\emptyset, 1)$ is the identity element of $\mathcal{S}(M)$, where $1$ is the identity element of $M$. If $(A,e)$ and $(B,f)$ are idempotents of $\mathcal{S}(M)$, then $(A,e)(B,f) = (A \cup eB, ef)$ and $(B,f)(A,e) = (B \cup fA, fe)$ need not be idempotent, and need not be equal.  If $1$ is the only idempotent of $M$, then $\mathcal{E}$ is the set of all idempotents of $\mathcal{S}(M)$ and, in any case, this is easily seen to be a semilattice since $(A,1)(B,1) = (A \cup B,1) = (B \cup A, 1) = (B,1)(A,1)$. Conversely, suppose that the set of idempotents of $\mathcal{S}(M)$ is a semilattice. For any $e,f \in E(M)$ note that $(\emptyset,e)$ and $(\emptyset,f)$ are idempotents, and since these commute we must have $ef=fe$. Thus $E(M)$ must be a semilattice. Next, $(\{ e\},ef)=(\{ e\},e)(\emptyset,f)=(\emptyset,f)(\{ e\},e)=(\{ fe\},fe)$. Dually, $(\{ f\},fe)=(\{ ef\},ef)$ so that since $fe=ef$ 
we obtain that $e=f$.

(2) Suppose $\mathcal{S}(M)$ is regular and let $x \in M$. By assumption there exists  $(Y,y) \in \mathcal{S}(M)$ with 
$$(\{1\},x)=(\{1\},x)(Y, y)(\{1\},x) = (\{1, xy\} \cup xY,xyx).$$
Thus $xy=1$ and as this holds for {\em every} $x\in M$ we deduce that $M$ is a group. Conversely, if $M$ is a group, then it is easy to see that $\mathcal{S}(M)$ has a particularly nice structure; it is an inverse monoid with $(X,x)^{-1} = (x^{-1}X, x^{-1})$. Further, since $\mathcal{S}(M)$ is a semidirect product of a semilattice by a group it follows from classical results that it is proper \cite{H}. It is instructive to see how this transpires: for any $(A,a), (B,b)\in \mathcal{S}(M)$ we have that
$(A,a)\,\sigma\,  (B,b)$ if and only if $a=b$, $(A,a)\,\mathcal{R}\,  (B,b)$ if and only if $A=B$ and  $(A,a)\,\mathcal{L}\,  (B,b)$ if and only if $a^{-1}A=b^{-1}B$. Finally, since $\mathcal{S}(M)$ is inverse, it is weakly coherent by Proposition \ref{prop:rabundweak}.
\end{proof}

We next consider the situations in which $\mathcal{S}(M)$ lies in the class of left abundant,  left $\mathcal{E}$-Ehresmann, left $\mathcal{E}$-restriction or left $\mathcal{E}$-ample monoids.

\begin{proposition}
\label{prop:Rrel}
Let $M$ be a monoid,  $\mathcal{E} = \{(F,1): F\subseteq M\}$ and $(A,a), (B,b) \in \mathcal{S}(M)$.
\begin{enumerate}[label=\textnormal{(\arabic*)}]
\item $(A, a) \,\mathcal{R}^*\, (B,b)$ if and only if $A=B$ and $a \, \mathcal{R}^*\,  b$. 
\item $(A, a) \,\widetilde{\mathcal{R}}\, (B,b)$ if and only if $A=B$ and $a \, \widetilde{\mathcal{R}}_{\hat{A}}\, b$.
\item $(A, a) \,\widetilde{\mathcal{R}}_{\mathcal{E}} \, 
(B,b)$ if and only if $A=B$.
\item $\mathcal{R}^* = \widetilde{\mathcal{R}}_{\mathcal{E}}$ in $\mathcal{S}(M)$ if and only if $M$ is right cancellative.
\item $\mathcal{S}(M)$ is left abundant if and only if for each $A \subseteq M$ and $a \in M$ there exists an idempotent $e \in \hat{A}$ with $a\, \mathcal{R}^* \,e$.
\item $\mathcal{S}(M)$ is always a proper left  $\mathcal{E}$-restriction monoid and so certainly left $\mathcal{E}$-Ehresmann;  it is  left $\mathcal{E}$-ample if and only if $M$ is right cancellative, in which case $\mathcal{E}=E(\mathcal{S}(M))$.
\end{enumerate} 
\end{proposition}
\begin{proof}
(1) Suppose first that $(A,a) \, \mathcal{R}^* \, (B,b)$. Since $(A,1)(A,a) = (\emptyset, 1)(A,a)$ and $(B,1)(B,b) = (\emptyset, 1)(B,b)$ it follows that $A=A\cup B=B$. If $ua=va$ then $(\emptyset,u)(A,a) = (\emptyset, v)(A,a)$ giving $(\emptyset,u)(B,b) = (\emptyset, v)(B,b)$, and hence $ub=  vb$. A dual argument shows that if $ub=vb$ then $ua=va$. Thus $a \, \mathcal{R}^*\,b$. Conversely, if $A=B$ and  $a \, \mathcal{R}^* \,b$ then it is easy to see that $(U,u)(A,a) = (V,v)(A,a)$, if and only if $(U,u)(B,b) = (V,v)(B,b)$.

(2) By definition $(A,a)$ and $(B,b)$ are $\widetilde{\mathcal{R}}$-related if and only if they have the same set of idempotent left identities. Suppose first that $(A, a)\, \widetilde{\mathcal{R}} \, (B,b)$. Since $(A,1)(A,a) = (A, a)$, it follows that we must also have $(A,1)(B,b) = (B,b)$ giving $A \subseteq B$; by a left-right dual argument we also have $B\subseteq A$, giving $A=B$ and hence $(A,a) \,\widetilde{\mathcal{R}}\,  (A,b)$ Consider now an idempotent of the form $(A,e)$ where $e \in \hat{A}$. Note that $(A, e)$ is a left identity of $(A,a)$ if and only if $ea = a$. Since $(A,a)\, \widetilde{\mathcal{R}}\,  (A,b)$ it then follows that each idempotent left identity of $a$ lying in $\hat{A}$ must be an idempotent left identity of $b$, and vice versa, giving that $a\,\widetilde{\mathcal{R}}_{\hat{A}}\,  b$. Conversely, suppose that $A=B$ and $a \,\widetilde{\mathcal{R}}_{\hat{A}}\, b$.  Let $(E,e)$ be an idempotent satisfying $(E,e)(A,a) =  (A,a)$, that is, $E\cup eA = A$ and $ea=a$. Since $eA \subseteq A$ we have $e \in \hat{A}$, and so $eb=b$. Now $(E,e)(A,b) = (E\cup eA, eb) = (A,b)$, as required.

(3) Restricting attention to the idempotents of $\mathcal{E}$, it is clear that $(E,1)$ is a left identity for $(A,a)$ if and only if $E \subseteq A$, and hence the set of idempotents of $\mathcal{E}$ acting identically on the left of $(A,a)$ will be equal to the corresponding set of idempotents for $(B,b)$ if and only if $A=B$.

(4)  Let $a \in M$ and consider $(\{1\}, a) \in \mathcal{S}(M)$. It is easy to see that the idempotent left identities of $(\{1\}, a)$ are  $(\emptyset, 1)$ and $(\{1\}, 1)$. Thus for all $a \in M$ we have $(\{1\}, a) \,\widetilde{\mathcal{R}}_{\mathcal{E}}\,  (\{1\}, 1)$. If $\mathcal{R}^* = \widetilde{\mathcal{R}}_\mathcal{E}$ it then follows from part (1) that $a \, \mathcal{R}^* \, 1$ for all $a \in M$. Thus $ua=va$ if and only if $u=v$, that is, $M$ is right cancellative. Conversely, if $M$ is right cancellative, then 
$E(M)=\{ 1\}$ so that $\mathcal{E}=E(\mathcal{S}(M))$ and $a\, \mathcal{R}^*\,  b$ holds for all $a,b \in M$. From parts (1) and (3) we have $\mathcal{R}^* = \widetilde{\mathcal{R}}=\widetilde{\mathcal{R}}_{\mathcal{E}}$.

(5) By part (1) the $\mathcal{R}^*$-class of $(A,a)$ contains an idempotent $(E,e)$ if and only if $A=E$, $e \in \hat{A}$ and $a \, \mathcal{R}^* \, e$.

(6) Let $(T,t) \in \mathcal{S}(M)$ and note  $(T,t) (A,a) = (T \cup tA, ta)$. From (3), if $(A,a) \,{\mathcal{R}}_{\mathcal{E}}\,  (B,b)$ then $A=B$ and hence also $(T,t) (A,a) \,\widetilde{\mathcal{R}}_{\mathcal{E}}\,  (T,t) (B,b)$, so that $\widetilde{\mathcal{R}}_{\mathcal{E}}$ is a left congruence. It is also clear that $(A,1)$ is the unique idempotent of $\mathcal{E}$ lying in the $\tilde{\mathcal{R}}_{\mathcal{E}}$-class of $(A,a)$, thus $\mathcal{S}(M)$ is left $\mathcal{E}$-Ehresmann. Writing $(A,a)^+ =(A,1)$ we see that $\mathcal{E}=\{(A,a)^+: (A,a) \in \mathcal{S}(M)\}$ and it is straightforward to check that the left ample identity holds:
\begin{eqnarray*}
(A,a)(B,b)^+ &=& (A,a)(B,1) = (A \cup aB,a) = (A \cup aB,1)(A,a)\\ &=& (A \cup aB,ab)^+(A,a) = 
((A,a)(B,b))^+(A,a). 
\end{eqnarray*}
Thus $\mathcal{S}(M)$ is left $\mathcal{E}$-restriction. Recall that $\sigma_{\mathcal{E}}$ is the least congruence identifying all idempotents of $\mathcal{E}$. Since all idempotents in $\mathcal{E}$ have second component $1$, it is easy to see that if $(A,a) \,\sigma_{\mathcal{E}}\,  (B,b)$ then we must have $a=b$. From part (3) we also have that  $(A,a)\,\widetilde{\mathcal{R}}_{\mathcal{E}}\, (B,b)$ if and only if $A=B$. Thus $\widetilde{\mathcal{R}}_{\mathcal{E}} \cap \sigma_{\mathcal{E}}$ is trivial and hence $\mathcal{S}(M)$ is proper left $\mathcal{E}$-restriction. By part (4) we have seen that $\widetilde{\mathcal{R}}_\mathcal{E} = \mathcal{R}^*$ if and only if $M$ is right cancellative; thus $\mathcal{S}(M)$ is proper left $\mathcal{E}$-ample if and only if $M$ is right cancellative.
\end{proof}

By Proposition \ref{prop:Rrel}, left abundance of $M$ is a necessary condition for $\mathcal{S}(M)$ to be left abundant. Likewise, by the left-right dual to Theorem \ref{thm:retract}, $M$ being finitely left equated is a necessary condition for $\mathcal{S}(M)$ to be finitely left equated. Neither condition is sufficient however, as we shall see in Example \ref{ex:labund}.

We now consider the situations in which $\mathcal{S}(M)$ lies in the class of right abundant, right $\mathcal{E}$-Ehresmann, right $\mathcal{E}$-restriction or right $\mathcal{E}$-ample monoids.

\begin{proposition}
\label{prop:Lrel}
Let $M$ be a monoid,  $\mathcal{E} = \{(E,1): E \subseteq M\}$ and $(A,a), (B,b) \in \mathcal{S}(M)$.
\begin{enumerate}[label=\textnormal{(\arabic*)}]
\item $(A, a) \,\widetilde{\mathcal{L}} \, (B,b)$ if and only if $a^{-1}A=b^{-1}B$ and $a \,\widetilde{\mathcal{L}}\,  b$.
\item $(A, a) \,\mathcal{L}^* \, (B,b)$ if and only if $a^{-1}A=b^{-1}B$ and $a \,\mathcal{L}^*\,  b$.

\item $(A, a) \,\widetilde{\mathcal{L}}_{\mathcal{E}} \,
(B,b)$ if and only if $a^{-1}A=b^{-1}B$.
\item $\mathcal{L}^* = \widetilde{\mathcal{L}}_{\mathcal{E}}$ in $\mathcal{S}(M)$ if and only if $M$ is left cancellative.
\item $\mathcal{S}(M)$ is right abundant if and only if $M$ is right abundant.
\item Each $\tilde{\mathcal{L}}_{\mathcal{E}}$--class of $\mathcal{S}(M)$ contains a unique idempotent from $\mathcal{E}$, and $\mathcal{S}(M)$ is a right $\mathcal{E}$-Ehresmann monoid (or right $\mathcal{E}$-adequate) if and only if $M$ is left cancellative,
in which case $\mathcal{E}=E(\mathcal{S}(M))$. 
\item  $\mathcal{S}(M)$ is  right ample if and only if $M$ is a group (in which case $\mathcal{S}(M)$ is a proper inverse monoid).
\end{enumerate} 
\end{proposition}

\begin{proof} 
(1) Suppose first that $(A,a)\, \widetilde{\mathcal{L}} \, (B,b)$. Since $(A,a)(a^{-1}A,1) = (A,a)$ we must have $(B,b)(a^{-1}A, 1) = (B,b)$, which implies that $b(a^{-1}A) \subseteq B$, which in turn implies that $a^{-1}A \subseteq b^{-1}B$. Thus with the dual argument we have that $a^{-1}A = b^{-1}B$. Next, suppose that $ae=a$. Since $(A,a)(\emptyset, e) = (A,e)$ we also require that $(B,b)(\emptyset, e) = (B,b)$, that is $be=b$, and it follows that $a \,\widetilde{\mathcal{L}}\, b$.

Conversely, suppose that $a^{-1}A = b^{-1}B$ and $a \,\widetilde{\mathcal{L}}\, b$, and let $(E,e)$ be an idempotent with $(A,a)(E,e) = (A,a)$. Since $A \cup aE =A$, we have that $aE \subseteq A$, or equivalently  $E \subseteq a^{-1}A = b^{-1}B$, giving $bE \subseteq B$. Further, since
$ae=a$ and $a \,\widetilde{\mathcal{L}}\, b$, we have $be=b$. Hence $(B,b)(E,e) = (B,be)=(B,b)$ and it follows that  $(A,a)\, \widetilde{\mathcal{L}} \, (B,b)$.

(2) Suppose first that $(A,a)\, \mathcal{L}^* \, (B,b)$. Then 
$(A,a)\, \widetilde{\mathcal{L}} \, (B,b)$ so that by (1) we have $a^{-1}A = b^{-1}B$. Next, suppose that $ax=ay$. Since $(A,a)(\emptyset, x) = (A,a)(\emptyset, y)$ we also require that $(B,b)(\emptyset, x) = (B,b)(\emptyset, y)$, that is $bx=by$. 
Thus with the  dual argument we obtain $ a \,\mathcal{L}^*\, b$.

Conversely, suppose that $a^{-1}A = b^{-1}B$ and $a \,\mathcal{L}^*\, b$.  Let $(A,a)(X,x) = (A,a)(Y,y)$. Then $A \cup aX = A \cup aY$ and $ax=ay$. From the latter we immediately deduce that $bx=by$. Meanwhile, the former gives $aX \setminus A = aY \setminus A$. Let $p\in B\cup bX$; we show that $p\in B\cup bY$. Clearly this is true if $p\in B$, so we assume $p\notin B$. Then $p=bu$ for some $u\in X$. If $au\in A$ then $u\in a^{-1}A=b^{-1}B$ so that $p\in B$, a contradiction. Thus $au\in aX\setminus A=aY\setminus A$, giving $au=av$ and then $p=bu=bv$ for some $v\in Y$, since $a\,\mathcal{L}^*\, b$. Thus in this case also $p\in B\cup bY$. Together with the dual argument we obtain $B\cup bX=B\cup bY$ 
from which it follows that $(B, b)(X,x) = (B,b)(Y,y)$.

(3) By definition $(A,a)$ and $(B,b)$ are $\tilde{\mathcal{L}}_{\mathcal{E}}$-related if and only if they have the same set of idempotent right identities in $\mathcal{E}$. Since $(A,a)(E,1) = (A,a)$ if and only if $aE \subseteq A$, or equivalently, $E \subseteq a^{-1}A$, and it follows from this that $(A,a)$ and $(B,b)$ have the same set of right identities in $\mathcal{E}$ if and only if $a^{-1}A=b^{-1}B$.

(4) It is immediate from (2) and (3) that $\mathcal{L}^* = \mathcal{L}_{\mathcal{E}}$ if and only if all elements of $M$ are $\mathcal{L}^*$-related or, in other words, if and only if $M$ is left cancellative.

(5) It is immediate from part (2)  that $(A,a) \, \mathcal{L}^* \, (E,e)$ if and only if $a \,\mathcal{L}^* \,e$ and $e^{-1}E = a^{-1}A$. Since $(E,e)$ being idempotent implies that $e$ is idempotent, $\mathcal{S}(M)$ right abundant implies that $M$ is right abundant. Conversely, suppose that $M$ is right abundant. Then for all $a \in M$ there exists an idempotent $e \in M$ such that for all $u,v \in M$, $au=av$ if and only if $eu=ev$. Since $e^2 = e$ note that $e$ must be an idempotent right identity for $a$. Now taking $E=a^{-1}A$ it is straightforward to check that $a^{-1}A = e^{-1}E$ (since $u \in a^{-1}A$  if and only if $au=aeu \in A$ if and only if $eu \in a^{-1}A = E$ if and only if $u \in e^{-1}E$) and  $(E,e)$ is idempotent (since $eE = e(a^{-1}A) = e((ae)^{-1}A) =e(e^{-1}(a^{-1}A) )=  e(e^{-1}E)\subseteq E$). It then follows from part (2) that $(A,a) \, \mathcal{L}^* \, (E,e)$. 

(6) It follows from the above that $(a^{-1}A, 1)$ is the unique element of $\mathcal{E}$ that is $\widetilde{\mathcal{L}}_{\mathcal{E}}$-related to $(A,a)$. Thus $\mathcal{S}(M)$ will be right $\mathcal{E}$-Ehresmann if and only if $\widetilde{\mathcal{L}}_{\mathcal{E}}$ is a right congruence. We show that this is the case if and only if $M$ is left cancellative; in which case by part (4) we also have that $\mathcal{L}^* =\widetilde{\mathcal{L}}_\mathcal{E}$ giving that $\mathcal{S}(M)$ is also right $\mathcal{E}$-adequate. First, suppose that $M$ is left cancellative. Then it is straightforward to check that for all $a \in M$ and all $T \subseteq M$ we have $a^{-1}(aT) = T$. Now suppose that $(A,a)\,\widetilde{\mathcal{L}}_{\mathcal{E}}\,(B,b)$ and let $(T,t) \in \mathcal{S}(M)$. Then $(A,a)(T,t) = (A \cup aT, at)$, $(B,b)(T,t) = (B \cup bT, bt)$ and by (3) we have $a^{-1}A = b^{-1}B$. Thus
$$(at)^{-1}(A \cup aT) = t^{-1} (a^{-1}A \cup a^{-1}(aT)) = t^{-1} (a^{-1}A \cup T) = t^{-1} (b^{-1}B \cup b^{-1}(bT)) = (bt)^{-1}(B \cup bT),$$
and we have $(A,a)(T,t)\,\widetilde{\mathcal{L}}_{\mathcal{E}}\,(B,b)(T,t)$. On the other hand, if $M$ is not left cancellative, then there exist elements $a,t,u$ such that $at=au$ but $t \neq u$. By the previous part we have that $(\emptyset, a) \,\widetilde{\mathcal{L}}_{\mathcal{E}}\, (\emptyset, 1)$. However, $(\{at\}, a)=(\emptyset, a)(\{t\}, 1)$ is not $\widetilde{\mathcal{L}}_{\mathcal{E}}$-related to $(\{t\},1) = (\emptyset, 1)(\{t\}, 1)$ since $u\in a^{-1}\{ at\}$ but $u\notin 1^{-1}\{ t\}$.

(7) From the previous part, we may define a unary operation $(A,a)^* = (a^{-1}A,1)$, and if $M$ is left cancellative then $\mathcal{S}(M)$ is right $\mathcal{E}$-Ehresmann with respect to this operation. However, $\mathcal{S}(M)$ need not be right $\mathcal{E}$-restriction. Indeed $(A,1)(B,b)=(A\cup B,b)$ but 
\[\begin{array}{rcl}
(B,b)((A,1)(B,b))^* &=&  (B,b)(A \cup B,b)^*\\
&=&  (B,b)(b^{-1}A\cup b^{-1}B,1) \\
&=&(B \cup b(b^{-1}A) \cup b(b^{-1}B),b) \\
&=& (B \cup b(b^{-1}A),b),
\end{array}\]
where the last step follows from the fact that for all $y\in M$ and $X\subseteq M$ we have that $y(y^{-1}X)\subseteq X$. 
From Remark~\ref{rem:inclusions}, 
 $M$ is a group if and only if this inclusion may always be replaced  by $y(y^{-1}X)=X$. Thus if $M$ is a group then   $b(b^{-1}A)=A$ and $(A,1)(B,b)=(B,b)((A,1)(B,b))^*$. Moreover, we have already seen (in Lemma \ref{lem:Sreg}) that if $M$ is a group then $\mathcal{S}(M)$ is a proper inverse monoid. Conversely, if $M$ is not a group and we choose $y\in M$ and $X\subseteq M$ with  $y(y^{-1}X) \subsetneq X$, then taking $A=X$, $B=\emptyset$ and $b=y$ in the above we see that the right ample identity fails. 
\end{proof}

\begin{remark}
We note that Lemma \ref{lem:Sreg} together with Proposition \ref{prop:Rrel} yields that the expansion $\mathcal{S}$ maps groups to proper inverse monoids, unipotent monoids to proper left restriction monoids, and right cancellative monoids to left ample monoids. Similarly, by Proposition \ref{prop:Lrel} we have that $\mathcal{S}$ maps right abundant monoids to right abundant monoids and left cancellative monoids to right adequate monoids. 
\end{remark}

\section{Finitary conditions}
In this section we ask: which monoids $M$ have the property that $\mathcal{S}(M)$ satisfies one of our suite of finitary conditions? In all cases, a necessary (but not sufficient) condition is that $M$ itself satisfies the same finitary property. 

\subsection{Conditions (R) and (L) }\label{sub:RL}
For condition (R) and condition (L), the characterisation turns out to be very simple:
\begin{proposition}\label{prop:finite}
Let $M$ be a monoid. The following are equivalent:
\begin{enumerate}
    \item $\mathcal{S}(M)$ satisfies condition (R);
    \item $\mathcal{S}(M)$ satisfies condition (L);
    \item $M$ is finite.
\end{enumerate} 
\end{proposition}
\begin{proof}
If $M$ is finite then $\mathcal{S}(M)$ is also finite and hence (R) and (L) hold.

To show that (1) implies (3), we prove the contrapositive. Suppose $M$ is infinite and consider the subact $R:=\mathbf{R}((M,1),(M,1))$. By definition 
\begin{eqnarray*}
R &=& \{ ((D,d), (E,e)): (M,1)(D,d)=(M,1)(E,e)\} = \{ ((D,d), (E,e)): M \cup D = M \cup E, d=e\} \\&=& \{ ((D,d), (E,d)): D, E \subseteq M\}.    
\end{eqnarray*}
We show that $R$ is not finitely generated as a subact of $\mathcal{S}(M)\times \mathcal{S}(M)$. Suppose for contradiction that $X \subseteq R$ is a finite generating set.  Let $b,c \in M$ be distinct elements such that $\{b\} \neq P$ and $\{c\} \neq Q$ for all $((P,p),(Q,p)) \in X$. Since $((\{b\},1), (\{c\},1)) \in R$, there must exist $((P,p),(Q,q)) \in X$ and $(T,t) \in \mathcal{S}(M)$ such that
\[\begin{array}{rcccl}(\{b\},1) &=&(P,p)(T,t) &=&(P \cup pT, pt)\\(\{c\},1) &=& (Q,p)(T,t) &=& (Q \cup pT, pt).
\end{array}\]
Since by assumption $\{b\} \neq P$ and $\{c\} \neq Q$, it follows that $\{b\} = pT = \{c\}$, giving the desired contradiction.

Similarly, to show that (2) implies (3), we prove the contrapositive. Suppose $M$ is infinite with cardinality $\kappa$ and consider the subact $L:=\mathbf{L}((M,1),(M,1))$.  By definition 
\begin{eqnarray*}
L &=& \{ ((D,d), (E,e)): (D,d)(M,1)=(E,e)(M,1)\} = \{ ((D,d), (E,e)): D \cup dM = E \cup eM, d=e\} \\&=& \{ ((D,d), (E,d)): D \cup dM = E \cup dM\}.    
\end{eqnarray*}
We show that $L$ is not finitely generated as a subact of $\mathcal{S}(M)\times \mathcal{S}(M)$. Suppose for contradiction that $X \subseteq L$ is a finite generating set. Let $$Y  =\{tP: t \in M, \mbox{ and there exists } ((P,p), (Q,p)) \in X \mbox{ or } ((Q,p), (P,p)) \in X\}.$$
Since $X$ is finite,  $|Y| \leq  \kappa < 2^\kappa = |\mathcal{P}(M)|$. So there exists $B \in \mathcal{P}(M)$ such that $B \not\in Y$. Since $((
B,1), (M \setminus B,1)) \in L$, there must exist $((P,p),(Q,q)) \in X$ and $(T,t) \in \mathcal{S}(M)$ such that

\[\begin{array}{rcccl}
(B,1) &=& (T,t)(P,p) &=& (T \cup tP, tp)\\(M \setminus B,1) &=& (T,t)(Q,p) &=& (T \cup tQ, tp).
\end{array}\]
 Since $B \cap (M \setminus B)=\emptyset$, it follows that $T= \emptyset$ and hence $B = tP \in Y$, giving the desired contradiction. 
\end{proof}

\begin{remark}\label{rem:Rr?} Note that in Proposition~\ref{prop:finite}, to ensure that $M$ is finite, it is enough to insist that either $\mathbf{R}((M,1),(M,1))$ or $\mathbf{L}((M,1),(M,1))$ is finitely generated as a right (respectively, left) act. This certainly implies that 
$\mathbf{r}((M,1))$ (respectively $\mathbf{l}((M,1))$) is finitely generated as a right (respectively, left) congruence. We will return to this point at the end of the paper.
    \end{remark}

\subsection{The ideal Howson properties}\label{sub:howson}

Since $M$ is a retract of $\mathcal{S}(M)$, it follows immediately from \cite[Corollary 3.5]{DGM24} and its left-right dual that if $\mathcal{S}(M)$ is right (respectively, left) ideal Howson, then $M$ is right (respectively, left) ideal Howson. On the other hand, if $M$ is left or right ideal Howson it need not be the case that $\mathcal{S}(M)$ is left or right ideal Howson, as we shall demonstrate below.  We give a stronger property which will turn out to exactly characterise the monoids $M$ for which $\mathcal{S}(M)$ is right ideal Howson.
\begin{definition}\label{def:srih}(Strongly right ideal Howson)
A  monoid $M$ is {\em strongly right ideal Howson} if for all $a,b \in M$ with $aM \cap bM \neq \emptyset$ there exists $n \geq 1$ and $u_1, \ldots, u_n \in aM \cap bM$ such that 
\begin{enumerate} [label=\textnormal{(\arabic*)}]
\item $aM \cap bM = \bigcup_{1 \leq i \leq n} u_iM$;
\item for each $i=1, \ldots, n$, the set $L_i = (aM \cap bM) \setminus u_iM$ is finite.
\end{enumerate}
If $M$ is strongly right ideal Howson and $aM \cap bM \neq \emptyset$ we say that elements $u_1, \ldots, u_n$ satisfying conditions (1) and (2) above are a {\em strong finite generating set} for $aM \cap bM$.  
\end{definition}

\begin{remark}(Principally right ideal Howson)
Notice that if $M$ is strongly right ideal Howson and moreover for all $a,b \in M$ with $aM\cap bM\neq \emptyset$ we have a strong finite generating set of cardinality $n=1$, then the first condition simply says that any non-empty intersection of principal ideals is again principal, whilst the second condition is vacuous (the difference is empty). In this case we shall say that $M$ is {\em principally right ideal Howson}.
\end{remark}

From the definitions it is clear that principally right ideal Howson implies strongly right ideal Howson which in turn implies right ideal Howson. The converse implications do not hold: 
\begin{example}\label{ex:notstrongright}The construction in \cite[Theorem 3.10]{CG21} may be invoked to give a finite band $B_2^1$ that is not principally right ideal Howson; since it is finite, it is certainly right ideal Howson. In a similar vein, the semigroup $S=S_{2,0}$ defined in Section 5.1 of \cite{CG21} with presentation
\[S=\langle a,b,w_1,w_2, v_1,v_2\mid aw_1=bv_1, aw_2=bv_2\rangle\]
is right ideal Howson and by \cite[Lemma 2.2]{CG21} so too is the monoid $M=S^1$; we claim $M$ is not strongly right ideal Howson. To this end, notice that 
\[[a]M\cap [b]M=[aw_1]M \cup [aw_2]M\]
and clearly 
\[\{ [aw_1a^n]:n\in \mathbb{N}\}\]
is an infinite subset of 
\[([a]M\cap [b]M)\setminus [aw_2]M.\]
\end{example}

\begin{theorem}\label{thm:srih}
The expansion $\mathcal{S}(M)$ of a monoid $M$ is right ideal Howson if and only if $M$ is strongly right ideal Howson.   Moreover,  $\mathcal{S}(M)$ is principally right ideal Howson if and only if $M$ is principally right ideal Howson.
\end{theorem}
\begin{proof}
In what follows we shall suppress the dependence on $M$ and write simply $S:=\mathcal{S}(M)$. Suppose first that $M$ is strongly right ideal Howson. Let $(A,a), (B,b) \in S$ be such that $(A,a)S \cap (B,b)S \neq \emptyset$. Then there exists $(X,x), (Y,y) \in S$ with $(A,a)(X,x) = (B,b)(Y,y)$, that is, $A \cup aX = B \cup bY$ and $ax=by$. In particular, we have that $B \setminus A \subseteq aM$, $A \setminus B \subseteq bM$  and $aM \cap bM \neq \emptyset$. Let $P, Q \subseteq M$ be such that $B \setminus A =aP$ and $A \setminus B = bQ$. Our assumption is that $M$ is strongly right ideal Howson, so there exists a  strong finite generating set $u_1, \ldots, u_n$ for $aM \cap bM$.  Since each $u_i \in aM \cap bM$ we may write $u_i=ap_i=bq_i$ for some $p_i,q_i \in M$.  Let $L_i$ be defined as in Definition~\ref{def:srih}. We now claim that the finite set
\[K:=\{(A \cup B \cup C, u_i): C \subseteq L_i, 1 \leq i \leq n\}\] is a generating set for $I:=(A,a)S \cap (B,b)S$. 

First note that if $C \subseteq L_i$ for some $i$, then in particular $C \subseteq aM \cap bM$ and so we may write $C=aC_a=bC_b$ for some subsets $C_a, C_b$ of $M$, from which it is clear that
\begin{eqnarray*}(A,a)(P \cup C_a, p_i) &=& (A \cup aP \cup aC_a, ap_i) = (A \cup (B\setminus A) \cup C, u_i) = (A \cup B \cup C, u_i), \mbox{and}\\
(B,b)(Q \cup C_b, q_i) &=& (B \cup bQ \cup bC_b, bq_i) = (B \cup (A\setminus B) \cup C, u_i) = (A \cup B \cup C, u_i),
\end{eqnarray*}
demonstrating that $K\subseteq I$. 

Now suppose that $(D,d) \in I$. Then there exist $(U,u), (V,v) \in S$ such that $(D,d) = (A,a)(U,u)=(B,b)(V,v)$. That is, $D=A \cup aU = B \cup bV$ and $d=au=bv$. Since $d \in aM \cap bM$ we also have $d=u_im$ for some $i$ with $1 \leq i \leq n$ and some $m \in M$. Notice that if $c \in D \setminus (A \cup B)$ then $c \in aU \cap bV \subseteq aM \cap bM = \bigcup_{1 \leq j \leq n} u_jM$. Now let 
$C = \{c \in D \setminus (A \cup B): c \not\in u_iM\} \subseteq L_i$, and (noting that  $(D \setminus (A \cup B)) \setminus C \subseteq u_iM$) let $D'$ be such that $u_iD' = (D \setminus (A \cup B)) \setminus C$.  Then 
$$(D,d) = (A \cup B \cup C, u_i)(D',m),$$
Thus $(D,d)$ is in the right ideal generated by $K$ and it follows that $I=KS$ is finitely generated.  

In the case where $M$ is principally right ideal Howson, it follows immediately from the description of the generating set $K$ given above that a non-empty intersection $(A,a)S \cap (B,b)S$ is equal to the principal ideal generated by $(A\cup B, u)$ where $aM \cap bM = uM$, and so $S$ is principally right ideal Howson.

Conversely, suppose that $S$ is right ideal Howson. As commented, \cite[Corollary 3.5]{DGM24} gives that a retract of a right ideal Howson monoid is right ideal Howson; the proof  of this result also shows that a retract of a principally right ideal Howson monoid is principally right ideal Howson. Thus  $M$ is right ideal Howson and principally right ideal Howson if $S$ is. Suppose  for contradiction that $M$ is not strongly right ideal Howson. Then there exist $a, b \in M$ such that $aM \cap bM \neq \emptyset$ is finitely generated, but has no strong finite generating set. Fix a generating set $u_1, \ldots, u_n$ for $aM \cap bM$ of $\mathcal{R}$-incomparable elements. For our given generating set, let us also assume that $L_i = (aM \cap bM) \setminus u_iM$ is infinite (as must be the case for some $i$, since no finite generating set is strong). Let us write $u_j=ap_j=bq_j$ for each $j$. For each $w \in L_i$ we may also write $w = u_jv$ for some $j \neq i$ and $v \in M$, from which it is clear that
$$(\{w\}, u_i) = (\emptyset, a)( \{p_jv\},p_i) = (\emptyset, b)( \{q_jv\},q_i) \in (\emptyset,a)S \cap (\emptyset, b)S.$$
Now let $N$ be a generating set for $J:=(\emptyset,a)S \cap (\emptyset, b)S$. With any $w$ as above,  $(\{w\}, u_i) \in J$, so there exists $(C,c) \in N$ and $(D,d) \in S$ such that $$(\{w\}, u_i) = (C,c)(D,d) = (C \cup cD, cd).$$
Since $(C,c) \in J$ we have in particular that $c \in aM \cap bM = \bigcup_{1 \leq j \leq n} u_jM$. Moreover, since $u_i=cd$ and the generators are $\mathcal{R}$-incomparable we must have $c \in u_iM$, and hence $cD \subseteq u_iM$.  But now, $w \not\in u_iM$ and $\{w\}= C \cup cD$ forces $C=\{w\}$.  Thus for every $w \in L_i$, any generating set of $J$ must contain an element of the form $(\{w\}, c)$ where $c \in aM \cap bM$, and since $L_i$ is infinite, this demonstrates that each generating set of $J$ is infinite, giving the desired contradiction. 
\end{proof}

Having characterised the monoids $M$  for which $\mathcal{S}(M)$ is right ideal Howson, we give below a characterisation of the monoids $M$ for which $\mathcal{S}(M)$ is left ideal Howson.  This is not quite as neat as the characterisation of right ideal Howson; one requires a condition concerning the left action of $M$ on arbitrary subsets of $M$.

\begin{definition}
(Left co-ordinated)  We say that $a, b \in M$ has {\em finite left co-ordinate system $C$ with respect to subsets $A, B \subseteq M$} if $C$ is a finite subset of the left subact $$\mathbf{L}(a,b) = \{(p,q): pa = qb\}$$ of $M \times M$ with the property that for all $x,y \in M$ with $xa=yb$ there exists $t \in M$ and $(p,q) \in C$ such that $xa=tpa=tqb=yb$ and $tpA \cup tqB \subseteq xA \cup yB$.
We say that $M$ is (n-){\em left co-ordinated} if  for all elements $a,b$ and all subsets $A, B \subseteq M$ there is a finite left co-ordinate system $C$ (with $|C|\leq n)$ for $a,b$ with respect the sets $A,B$.
\end{definition}

\begin{remark}
\label{rem:leftcoord}

\begin{enumerate}

\item If condition (L) holds, then we also  have $M$ is left co-ordinated: taking $C$ to be a finite generating set for the $M$-act $\mathbf{L}(a,b)$, if $xa=yb$, then $(x,y) = t\cdot (p,q)$ for some $(p,q) \in C$ and $t \in M$ hence giving $xa=tpa=tqb=yb$ and $tpA \cup tqB = xA \cup yB$ for all sets $A, B \subseteq M$.
\item If $M$ is left co-ordinated, then it is in particular left ideal Howson: if $C$ is a finite left co-ordinate system for the pair $a,b$ with respect to (any choice of) sets $A$, $B$ then for any $f=xa=yb \in Ma \cap Mb$ we have $f=tpa=tqb$ for some $t \in M$ and $(p,q) \in C$, thus $D=\{pa: (p,q) \in C\}$ is a finite generating set for $Ma \cap Mb$. 
\item It follows from (1) that the property of being left co-ordinated is a finitary condition, since this is true of condition (L).
\item Note that $\mathbf{L}(a,a)$ is a left congruence, namely the  left annihilator congruence of $a$, and being finitely generated as a left $M$-act certainly implies being finitely generated as a left congruence. Thus if condition (L) holds (i.e. each left $M$-act $\mathbf{L}(a,b)$ is finitely generated) we have in particular that $M$ is finitely left-equated (i.e. each left annihilator congruence $\mathbf{l}(a)$ is finitely generated).
\item From (1) and (2) we have: 
$$M \mbox{  satisfies condition (L)} \Rightarrow  M \mbox{ is left co-ordinated} \Rightarrow M \mbox{ is left ideal Howson}.$$
\end{enumerate}
\end{remark}
For right cancellative monoids these three properties coincide. \begin{proposition}
\label{rem:leftcoordcancel}
Let $M$ be a monoid with the property that for all $a,b, f \in M$ the pair of equations $xa=f$, $yb=f$ has at most one solution $(x,y) \in M \times M$. Then: $M$ is left ideal Howson if and only if $M$ is left co-ordinated if and only if  $M$ satisfies (L).
\end{proposition}
\begin{proof}
From the fifth point of Remark \ref{rem:leftcoord} it suffices to show that if $M$ is left ideal Howson then $M$ satisfies (L).  Indeed, if $X$ is a finite generating set for $Ma \cap Mb$, first note that for each $g \in X$ there exists a unique pair $(p_g, q
_g) \in M \times M$ such that $g = p_ga=q_gb$, and so taking $C = \{(p_g,q_g): g\in X\}$ we see that $C\subseteq \mathbf{L}(a,b)$. If  $f=xa=yb$ holds, then $f=tg$ for some $t \in M$ and $g\in X$, and so $f=tp_ga=tq_gb$. Then by uniqueness of solution we have $x=tp_g$ and $y=tq_g$, demonstrating that $C$ is a finite generating set for $\mathbf{L}(a,b)$.  
\end{proof}
However, in general, this is not the case, as the next two examples demonstrate.

\begin{example}\label{ex:notstrongleft} 
We make use of \cite[Example 3.8]{GHS} to demonstrate that a left ideal Howson monoid need not be left co-ordinated. Consider  the four element diamond lattice $\mathbb{D}=\{E, F, D, G\}$ where $E \wedge F = G$ and $E \vee F = D$, and let $S$ be the lattice obtained as the direct product
$$S=\mathbb{N} \times \mathbb{D} = \{ E_i, F_i, D_i, G_i: i \geq 0\},$$
regarded as a semilattice with respect to $\wedge$ and then let $M= S^1$ be the monoid obtained by adjoining an identity to $S$. Since $M$ is a semilattice, it is in particular left ideal Howson. We show that the pair $a = E_0$, $b=F_0$ does not have a finite left co-ordinate system with respect to the sets $A= B = \{1\}$. First note that 
$$\mathbf{L}(a,b) = \{(F_i, E_j), (F_i, G_j), (G_i, E_j), (G_i, G_j): i , j \geq 0\}.$$
Let $C$ be a finite subset of $\mathbf{L}(a,b)$ and let $m, n \in \mathbb{N}$ be maximal such that $(F_m, Y) \in C$ for some $Y\in M$ and $(X, E_n) \in C$ for some $X \in M$. If $C$ were a left co-ordinate system for $a, b$ with respect to $A$ and $B$, then, in particular, for each $i>m$ and $j>y$ there must exist $(p,q) \in C$ and $t \in M$ such that $F_iE_0=tpE_0 = tqF_0 = E_jF_0$ and $\{tp, tq\} \subseteq \{F_i, E_j\}$. For $tp=F_i$, we require  $p \geq F_i>F_m$, and  contradicting the maximality of $m$. For $tp = E_j$ we require $p \geq E_j$, contradicting $(p,q) \in \mathbf{L}(a,b)$.
\end{example}

\begin{example}Let $M= \mathbb{Z} \cup \{\bot, \top\}$ ordered by $\bot < x < \top$ for all $x \in \mathbb{Z}$. Clearly $M$ is a monoid semilattice with operation $ab = {\rm min}(a,b)$, and identity element $\top$. We show that $M$ is left co-ordinated but does not satisfy $(L)$. Let $a, b\in M$, and $A, B \subseteq M$ and set 
$$C = \begin{cases}
\{(\top, \top)\}  & \mbox{ if } a=b\\
\{(\top, a), (a,a)\}  & \mbox{ if } a<b\\
\{(b, \top), (b,b)\}  & \mbox{ if } b<a.
\end{cases}$$
We claim that $C$ is a finite left co-ordinate system for $a, b$ with respect to $A, B$. To see this, we consider each of the three cases above ($a=b$, $a<b$ and $b<a$) in turn.

(1) If $a=b$ we have
$$C = \{(\top, \top)\} \subseteq \mathbf{L}(a,a) = \{(x,y): xa=ya\} = \{(x,y): x,y \geq a\} \cup \{(x,x): x<a\}.$$
We aim to show that for all $(x,y) \in \mathbf{L}(a,a)$ there exists $t \in M$ such that $xa = ta = ya$ and $tA \cup tB \subseteq xA \cup yB$. If $A = \emptyset$ we can take $t=y$; if $B = \emptyset$ we can take $t=x$; and if $x=y$ we can take $t = x$.  Thus we may assume without loss of generality that $A$ and $B$ are non-empty and $a \leq x <y$ (a symmetric argument will hold for $y<x$). We subdivide into further cases. First, if all $\alpha \in A$ satisfy $\alpha <a$, then taking $t=y>x \geq a$ yields $tA \cup tB = yA \cup yB = A \cup yB = xA \cup yB$, since both $x$ and $y$ are strictly greater than all elements of $A$. 
Second, if all $\alpha \in A$ satisfy $\alpha <x$ and there exists $\alpha' \in A$ with $a \leq \alpha' < x$, then taking $t=\alpha'>x \geq a$ yields 
\begin{eqnarray*}
tA \cup tB &=& \alpha' A \cup \alpha' B = \{\alpha' \alpha: \alpha \in A\} \cup \{\alpha' \beta: \beta \in B\}\\  &=& \{\alpha'\} \cup \{\gamma \in A: \gamma <\alpha'\}\cup \{\delta \in B: \delta <\alpha'\}\\
&\subseteq & \{\alpha'\} \cup \{\gamma \in A: \gamma <x\}\cup \{\delta \in B: \delta <y\} = xA \cup yB
\end{eqnarray*}
where the final line follows from the fact that $x$ and $y$ are strictly greater than  $\alpha' \in A$. Finally, if there exists $\alpha'' \in A$ with $\alpha'' \geq x$, then taking $t=x$ gives  
\begin{eqnarray*}
tA \cup tB &=& x A \cup x B = \{x \alpha: \alpha \in A\} \cup \{x \beta: \beta \in B\}\\  &=& \{x\} \cup \{\gamma \in A: \gamma <x\}\cup \{\delta \in B: \delta <x\}\\
&\subseteq & \{x\} \cup \{\gamma \in A: \gamma <x\}\cup \{\delta \in B: \delta <y\} = xA \cup yB
\end{eqnarray*}
where the final line follows from the fact that $x<y$.

(2) If $a<b$ we have
\begin{eqnarray*}
C = \{(\top, a), (a,a)\} \subseteq \mathbf{L}(a,b) &=& \{(x,y): xa=yb\}\\ &=& \{(x,a): x\geq a\} \cup \{(x,x): x<a\}\\
&=& \{x(\top,a): x\geq a\} \cup \{x(a,a): x<a\},
\end{eqnarray*}
showing that $\mathbf{L}(a,b)$ is finitely generated by $C$, and hence (by Remark \ref{rem:leftcoord}) $C$ is a finite left co-ordinate set for $a, b$ with respect to $A, B$.

(3) Dually, if $b<a$ we have
\begin{eqnarray*}
C = \{(b, \top), (b,b)\} \subseteq \mathbf{L}(a,b) &=& \{(x,y): xa=yb\}\\ &=& \{(b,y): y\geq b\} \cup \{(y,y): y<b\}\\
&=& \{y(b,\top): y\geq b\} \cup \{y(b,b): y<b\},
\end{eqnarray*}
showing that $\mathbf{L}(a,b)$ is finitely generated by $C$, and hence (by Remark \ref{rem:leftcoord}) $C$ is a finite left co-ordinate set for $a, b$ with respect to $A, B$.

Thus $M$ is finitely left co-ordinated. Finally, to see that $M$ does not satisfy (L), suppose for contradiction that $D$ is a finite generating set for the left $M$-act 
$$\mathbf{L}(\bot, \bot) = \{(x,y): x \bot = y \bot\} = M \times M.$$
Since, in particular $(\top, m) \in \mathbf{L}(\bot, \bot)$ for all $m \in M$, taking $m$ so that it does not occur as a second component of any pair in $D$ gives a contradiction, since we require that $(\top, m) = t(p,q)$ for some $t \in M$ and $(p,q) \in D$, but then $\top = tp$ implies $t=p=\top$, and then $q=m$.
\end{example}

\begin{thm}
\label{prop:leftcoord} The expansion $\mathcal{S}(M)$ of a monoid $M$  is left ideal Howson if and only if $M$ is left co-ordinated. Further, $\mathcal{S}(M)$ is principally left ideal Howson if and only if $M$ is left $1$-co-ordinated.
\end{thm}

\begin{proof}
We again suppress the dependence on $M$ and write simply $S:=\mathcal{S}(M)$.
Suppose first $M$ is left co-ordinated. Let $J:=S(A,a) \cap S(B,b) \neq \emptyset$, where  $(A,a), (B,b) \in S$. We shall show that $J$ is finitely generated. 
For each $(F, f) \in J$ there exist $(X,x), (Y,y) \in S$ such that
\[\begin{array}{rcccl}
(F,f) &=& (X,x)(A,a) &=& (X \cup xA,xa)\\
&=& (Y,y)(B,b) &=& (Y \cup yB,yb).
\end{array}\]
Thus $f=xa=yb$ and $xA \cup yB \subseteq F$ holds. Since $M$ is left co-ordinated, $a,b$ has a finite left co-ordinate system $C$ with respect to the sets $A,B$, and so by assumption we have that there exists $t \in M$ and $(p,q) \in C$ such that $f=tpa=tqb$ and $tpA \cup tqB \subseteq xA \cup yB$. Since $(p,q) \in C \subseteq \mathbf{L}(a,b)$, we have $pa=qb$. Setting $u=pa=qb$ we see that$$(F,f) = (F, t)(pA \cup qB, u).$$
Noting that
$$(pA \cup qB, u) = (qB,p)(A,a) = (pA,q)(B,b) \in J,$$
it then follows that $D:=\{(pA \cup qB, u): (p,q) \in C, u=pa=qb\}$ is a finite generating set for $J$ with $1 \leq |D| \leq |C|$. Thus if $M$ is left ($1$-)co-ordinated, then $\mathcal{S}(M)$ is (principally) left ideal Howson.

Conversely, suppose that $S$ is left ideal Howson. Let $a,b\in M$ with $Ma\cap Mb\neq\emptyset$ and let $A,B\subseteq M$.  We must find a left co-ordinate system of $Ma\cap Mb$ with respect to the sets $A,B$. Let $c,d\in M$ with $ca=db$ and  consider $(A,a), (B, b) \in S$. Notice that 
\[(dB,c)(A,a)=(dB\cup cA,ca)=(cA\cup dB,db)=(cA,d)(B,b),\]
so $J:=S(A,a) \cap S(B,b) \neq \emptyset$. Let  $\{(U_i, u_i): 1 \leq i \leq n\}$ be  a finite generating set for $J$. Then for each $i$ we must have $u_i=p_ia=q_ib$ for some $p_i, q_i \in M$, and $p_iA \cup q_i B \subseteq U_i$. Now suppose $f=xa=yb$ for some $x,y \in M$. Then
$$(xA \cup y B, f)  = (xA \cup yB,x)(A,a)=(xA \cup y B,y)(B,b)\in J,$$
and so by definition there exists $1 \leq i \leq n$ and $(T,t) \in S$ such that 
$$(xA \cup yB, f) = (T,t)(U_i, u_i) = (T \cup tU_i, tu_i).$$
Thus, $f=tu_i=tp_ia=tq_ib$ and $tp_iA \cup tq_iB \subseteq tU_i \subseteq xA \cup yB$
so that $C=\{(p_i, q_i): 1\leq i\leq n\}$
is a left co-ordinate system for $a,b$ with respect to the set $A,B$. If $\mathcal{S}(M)$ is (principally) left ideal Howson we see that $M$ is left (1-)co-ordinated.
\end{proof}

For right cancellative monoids, we obtain the following simplified characterisation. \begin{corollary}
Let $M$ be a right cancellative monoid. The expansion $\mathcal{S}(M)$ is left ideal Howson if and only if $M$ is left ideal Howson. 
\end{corollary}
\begin{proof}
This follows immediately from Proposition \ref{rem:leftcoordcancel} and \ref{prop:leftcoord}.
\end{proof}

To add to the list of properties in Theorem~\ref{thm:retract} we now have:
\begin{proposition}
If $M$ is a left (1-)co-ordinated monoid, and $T$ is a retract of $M$, then $T$ is also left (1-)co-ordinated.
\end{proposition}
\begin{proof}
Let $\varphi: M \rightarrow T$ be a retraction, and let $a,b \in T \subseteq M$, and $A, B \subseteq T \subseteq M$. Since $M$ is left co-ordinated, there exists a finite left co-ordinate system $C \subseteq M \times M$ for $a,b$ with respect to $A, B$ considered within $M$. In particular, for each  $(x,y) \in T \times T$ with $xa=yb$, there exists $(p,q) \in C$ and $t \in M$ such that $xa=tpa=tqb=yb$ and $tpA \cup tqB \subseteq xA \cup yB$. Applying the morphism $\varphi$ (recalling that this fixes all elements of $T$, and that $a,b,x,y, A, B$ all lie in $T$) yields that $xa=(t\varphi)( p\varphi  )a=(t\varphi)( q\varphi) b=yb$ and $(t\varphi)( p \varphi) A \cup (t\varphi)( q\varphi) B \subseteq xA \cup yB$. Further, $pa=qb$ gives that $(p\varphi)a=(q\varphi)b$ so that  $C' = \{(p\varphi, q\varphi): (p,q) \in C\}$ is a (finite) left co-ordinate system with respect to  $A,B$ considered within $T$. Clearly, if $|C|=1$, then $|C'|=1$ also.
\end{proof}

\subsection{The finitely equated conditions}\label{sub:fre}
Since $M$ is a retract of $\mathcal{S}(M)$ it follows from \cite[Corollary 4.12]{DGM24} that a necessary condition for $\mathcal{S}(M)$ to be finitely right (respectively, left) equated is that $M$ is finitely right (respectively, left) equated, however, as we shall see below, the converse does not hold (in either case). We look first at the case of finitely left equated monoids, providing an exact characterisation of the monoids $M$ for which $\mathcal{S}(M)$ is finitely left equated.

Recalling the left-right dual of Definitions~\ref{defn:rac} and \ref{defn:R}, we note
that for a monoid $M$ and $a\in M$ we have  $\mathbf{l}(a)=\{ (u,v)\in M\times M:ua=va\}=\mathbf{L}(a,a)$. It therefore makes sense to consider whether   $\mathbf{l}(a)$ is finitely generated as a subact of the left diagonal act of $M$, and this is easily seen to imply that $\mathbf{l}(a)$ is finitely generated as a left congruence. (The converse is not true, as we demonstrate in Example~\ref{ex:labund}.)

\begin{theorem}
\label{thm:FLE}
The  expansion $\mathcal{S}(M)$ of a monoid $M$  is finitely left equated if and only if for each $a \in M$ we have that $\mathbf{l}(a)$ is finitely generated as a left $M$-act.
\end{theorem}

\begin{proof}
Suppose first that for all $a \in M$ we have that $\mathbf{l}(a)$ is finitely generated as a left $M$-act. For an arbitrary $A \subseteq M$ and $a \in M$, we aim to show that $\mathbf{l}((A,a))$ is a finitely generated left congruence. By assumption, $\mathbf{l}(a)$ is finitely generated as an $M$-act; let $X \subseteq \mathbf{l}(a)$ be a finite symmetric generating set, so that $\mathbf{l}(a) = MX$. We claim that $\mathbf{l}((A,a))$ is the left congruence $\lambda_Y$ generated by the finite set $Y = \{((\emptyset, 1), (A,1))\} \cup \{((pA, q), (qA,p)): (p,q) \in X\}$. To see this, first note that $Y \subseteq \mathbf{l}((A,a))$ since:
\begin{eqnarray*}
(\emptyset, 1)(A,a) &=& (A,a) =  (A, 1)(A,a)\\
(pA, q)(A,a) &=& (pA \cup qA,qa) = (qA \cup pA,pa) =(qA,p)(A,a),
\end{eqnarray*}
where the second equality holds for all $(p,q) \in X \subseteq \mathbf{l}(a)$. Now for $((U,u), (V,v)) \in \mathbf{l}((A,a))$ we have $U \cup uA = V\cup vA $ and $ua=va$. Since $(u,v) \in \mathbf{l}(a)$, there exist $(p,q) \in X$ and $t \in M$ such that $u=tp$ and $v=tq$. Setting $W=U \cup uA = V\cup vA$ we have:
\begin{eqnarray*}
(U,u) = (U,u)(\emptyset, 1)\, \lambda_Y  \,(U,u)(A, 1) = (W,u) = (W,1)(vA,u) = (W,1)(tqA, tp) = (W,t)(qA, p)\\
(V,v) = (V,v)(\emptyset, 1)\, \lambda_Y  \,(V,v)(A, 1) = (W,v) = (W,1)(uA,v) = (W,1)(tpA, tq) = (W,t)(pA, tq),
\end{eqnarray*}
and since $(qA, p)\, \lambda_Y\,  (pA, q)$ this demonstrates that $(U,u)\, \lambda_Y \, (V,v)$.

Conversely, suppose that $\mathcal{S}(M)$ is finitely left equated. In particular for all $a \in M$ we have that $\mathbf{l}((\{1\}, a)$ is a finitely generated left congruence; let $X$ be a finite symmetric generating set. Now, if $((P,p), (Q,q)) \in X$ we have 
$$(P,p)(\{1\}, a) = (P \cup \{p\}, pa) = (Q \cup \{q\}, qa) = (Q,q)(\{1\}, a),$$
that is, $P \cup \{p\} = Q \cup \{q\}$  and $(p,q) \in \mathbf{l}(a)$. We claim that $\mathbf{l}(a) = MZ$ where 
$$Z = \{(1,1)\} \cup \{ (p,q): \exists ((P,p), (Q,q)) \in X\} \subseteq \mathbf{l}(a).$$ Since $X$ is finite, $Z$ is necessarily finite. Let $(u,v) \in \mathbf{l}(a).$ If $u=v$ then clearly $(u,v) = u(1,1)$. Suppose then that $u \neq v$. Then
$$(\{u\},v)(\{1\},a) = (\{u,v\}, va) = (\{u,v\}, ua) = (\{v\},u)(\{1\},a),$$
demonstrating that $((\{u\},v), (\{v\},u)) \in \mathbf{l}(\{1\}, a)$. Thus there is an $X$-sequence
$$(\{u\},v) = (T_1, t_1)(P_1, p_1), (T_1, t_1)(Q_1, q_1) = (T_2, t_2)(P_2, p_2), \ldots, (T_n, t_n)(Q_n, q_n) = (\{v\}, u).$$
For $1 \leq i \leq n$, set $Z_i = T_i \cup t_iP_i$, $z_i = t_ip_i$ and also set $Z_{n+1} = T_n \cup t_nQ_n$ and $z_{n+1} = t_nq_n=u$.   We prove (by induction) that for each $i$ we have $Z_i \subseteq \{u,v\}$ and $z_i \in \{u,v\}$. In view of our $X$-sequence, this is immediate for $i=1$. Now suppose that $1 \leq i \leq n$ is such that $Z_i \subseteq \{u,v\}$ and $z_i \in \{u,v\}$. Since $((P_i,p_i), (Q_i, q_i)) \in X$ we have $P_i\cup \{p_i\} = Q_i \cup \{q_i\}$ and so
$$\{u,v\} \supseteq Z_i \cup \{z_i\} = T_i \cup t_iP_i \cup \{t_ip_i\} = T_i \cup t_i(P_i \cup \{p_i\}) = T_i \cup t_i(Q_i \cup \{q_i\}) \supseteq T_i \cup t_iQ_i.$$
Thus
$Z_{i+1} = T_{i+1} \cup t_{i+1}P_{i+1} = T_i \cup t_iQ_i \subseteq \{u,v\}$
and $z_{i+1} = t_{i+1}p_{i+1} = t_iq_i \in \{u,v\}$. Now, returning to our $X$-sequence, we have (from the second co-ordinate)
$$v=t_1p_1, t_1q_1=t_2p_2, \ldots, t_nq_n=u$$
where (by the argument above) each term in this sequence, $z_i$, is either equal to one of $u$ or $v$. Since $z_1=v \neq u = z_{n+1}$, there must exist $i$ such that $z_i=t_ip_i = v$ and $z_{i+1}=t_{i+1}p_{i+1} = t_iq_i = u$, giving $(u,v) = t_i(q_i,p_i) \in MZ$ as required. 
\end{proof}
We state the following immediately consequence of Theorem \ref{prop:leftcoord} and Theorem \ref{thm:FLE}.
\begin{corollary}
The  expansion $\mathcal{S}(M)$ of a monoid $M$  is weakly left coherent if and only if $M$ is left co-ordinated and, for each $a \in M$, $\mathbf{l}(a)$ is finitely generated as a left $M$ act.  
\end{corollary}

Right cancellative monoids provide a particular instance of Theorem \ref{thm:FLE}, giving:

\begin{corollary}\label{cor:rcanc} Let $M$ be a right cancellative monoid. Then $\mathcal{S}(M)$ is weakly left coherent if and only if $M$ is left ideal Howson.
\end{corollary}

\begin{proof}
Suppose $M$ is right cancellative. Then for $a \in M$ we have 
$$\mathbf{l}(a) = \{(u,v): ua=va\} = \{(u,u): u \in M\} = M(1,1),$$
and so by Theorem \ref{thm:FLE}, $\mathcal{S}(M)$ is finitely left equated, and so $\mathcal{S}(M)$  is weakly left coherent if and only if $\mathcal{S}(M)$ is left ideal Howson. By Proposition~\ref{prop:leftcoord} this is equivalent to $M$ being left co-ordinated, and since $M$ is right cancellative Proposition \ref{rem:leftcoordcancel} gives that this is equivalent to the condition that $M$ is left ideal Howson.
\end{proof}

As we have alluded to earlier, there exist left abundant monoids $M$ such that $\mathcal{S}(M)$ is not left abundant; making use of Theorem \ref{thm:FLE} we exhibit an example of a semilattice $M$ (which is certainly left abundant) such that  $\mathcal{S}(M)$ is not finitely left equated and so certainly not left abundant.

\begin{example}\label{ex:labund} Let $M$ be the monoid with presentation {\em as a monoid semilattice} given by
$$M:= \langle a,b_i,c_i, i\in\mathbb{N}\mid b_ia=c_i a \rangle.$$
Since $M$ is a semilattice, it is certainly abundant and so finitely left equated. By construction, we have infinitely many distinct pairs $(b_i,c_i) \in \mathbf{l}(a)$, and it is easily seen that $\mathbf{l}(a)$ is not finitely generated as a left $M$-act. By Theorem \ref{thm:FLE}, we conclude that $\mathcal{S}(M)$ is not finitely left equated (and in particular not left abundant).  
\end{example}

In contrast with the previous example, and turning now to the case of finitely right equated monoids, it can be deduced from Proposition \ref{prop:rabundweak} and Proposition \ref{prop:Lrel} that if $M$ is right abundant, then $\mathcal{S}(M)$ \emph{is} finitely right equated. We generalise this to a statement about \emph{skeletons} below. 
\begin{definition}
For a right congruence $\rho_W$ on a monoid $M$ generated by a finite set $W$, and for $a,b \in \rho_W$, we may call the sequence $c_1,d_1,\cdots,c_n,d_n$ appearing in the $W$-sequence in Proposition~\ref{prop:gen} a {\em skeleton} from $a$ to $b$ over $W$. We say that $\rho_W$ {\em admits a fixed skeleton over $W$} if  there is a fixed sequence $c_1,d_1,\cdots,c_n,d_n$ that is a skeleton from $a$ to $b$ over $W$  for all $(a,b) \in \rho_W$. Further, we say that $M$ is \emph{strongly finitely right equated} if for all $x \in M$, the right annihilator congruence $\mathbf{r}(x)$ admits a fixed skeleton over some finite generating set.  
\end{definition}

Notice that several natural classes of monoids are strongly finitely right equated.  In particular, every right abundant monoid $M$ is strongly finitely right equated, since for any $x \in M$ and any $a,b \in \mathbf{r}(x)$, taking $e$ to be an idempotent in the $\mathcal{L}^*$-class of $x$ we have that $xe=x$ so that $Y:=\{(1,e), (e,1)\} \subseteq \mathbf{r}(a)$ and then for any $(u,v) \in \mathbf{r}(x)$ we see that
$$u=1u,\, eu = ev,\, 1v=v$$
is a $Y$-sequence, demonstrating that $\mathbf{r}(x)$ has fixed skeleton $1, e, e, 1$.
Our next aim is to show that $\mathcal{S}(M)$ is strongly finitely right equated if and only if $M$ is strongly finitely right equated, hence giving a generalisation of the result for right abundance. First notice that the property of being strongly finitely right equated is preserved under taking retracts.

\begin{lemma}
\label{lem:retract}
Let $M$ be a monoid and let $S$ be a retract of $M$. If $M$ is strongly finitely right equated, then so is $S$.
\end{lemma}

\begin{proof}
Let  $\theta: M \rightarrow S$ be a retraction where we consider $S$ as a submonoid of $M$. Fix $a \in S $ and let $\mathbf{r}_M(a)$ and $\mathbf{r}_S(a)$ denote the right annihilator congruence of $a$ on $M$ and $S$, respectively. By assumption, $\mathbf{r}_M(a)$ is strongly finitely right equated and so there is a finite generating set $X$ such that $\mathbf{r}_M(a)$ admits the finite skeleton $c_1, d_1, \ldots, c_n, d_n$ over $X$. Then, for $(u,v) \in \mathbf{r}_S(a) \subseteq \mathbf{r}_M(a)$ we have that there exist $t_1, \ldots, t_n\in M$ such that 
$u = c_1t_1, \, d_1t_1=c_2d_2,\,  \ldots, \,d_nt_n = v$. Applying the morphism $\theta$ (which fixes each element of $S$) then gives $u = c_1\theta t_1\theta, \,\,d_1\theta t_1\theta =c_2d_2\theta,\,  \ldots, \,d_n\theta t_n\theta = v$. Thus, noting that if $ap=aq$ for $p,q \in M$ we also have $a(p\theta)=a(q\theta)$ where $p,q \in S$, we may take $Y = \{(p\theta, q\theta): (p,q) \in X\}$ and sequence  $c_1\theta, d_1\theta, \ldots, c_n\theta, d_n\theta$ to see that $S$ is strongly finitely right equated.
\end{proof}

\begin{theorem}
\label{thm:sfre} The expansion $\mathcal{S}(M)$ of a monoid $M$ is strongly finitely right equated if and only if $M$ is strongly finitely right equated.
\end{theorem}

\begin{proof}
One direction follows immediately from Lemma \ref{lem:retract} since $M$ is a retract of $\mathcal{S}(M)$.

Conversely, suppose that $M$ is strongly finitely right equated. Fix $a \in M$ and suppose that $X$ is a finite generating set for $\mathbf{r}(a)$ with fixed skeleton $p_1, q_1, \ldots, p_n, q_n$. For an arbitrary $A \subseteq M$ we aim to show that $\mathbf{r}((A,a))$ has the following fixed skeleton: $$(\emptyset, 1), (a^{-1}A, 1), (a^{-1}A, 1), (\emptyset, 1), (\emptyset, p_1), (\emptyset, q_1), \ldots, (\emptyset, p_n), (\emptyset, q_n), (\emptyset, 1), (a^{-1}A, 1), (a^{-1}A, 1), (\emptyset, 1).$$
As a first step, observe that 
\begin{eqnarray*}
(A, a)(\emptyset, 1) = (A, a) = (A, a)(a^{-1}A, 1)\\
(A, a)(\emptyset, p_i) = (A, ap_i) = (A, aq_i)  = (A,a)(\emptyset, q_i),
\end{eqnarray*}
so that the proposed sequence indeed splits into pairs of consecutive elements belonging to $\mathbf{r}((A,a))$.

Let $((U,u), (V,v)) \in \mathbf{r}((A,a))$. Then $A \cup aU = A \cup aV$ and $au=av$. Write $U = U_1 \cup U_2$ and $V = V_1 \cup V_2$ where $U_1 = U \setminus a^{-1}A$, $U_2 = U \cap a^{-1}A$, $V_1 = V \setminus a^{-1}A$ and $V_2 = V \cap a^{-1}A$. Then the condition $A \cup aU = A \cup aV$ becomes $A \cup aU_1 = A \cup aV_1$, and since the union is disjoint we deduce that $a U_1 = aV_1$, and so in particular $U_1 = \emptyset$ if and only if $V_1 = \emptyset$. Also note that using the first two pairs in our proposed skeleton we obtain:
$$(U,u) = (\emptyset, 1)(U,u), \qquad  (a^{-1}A, 1)(U,u) = (a^{-1}A, 1)(U_1,u),  \qquad (\emptyset, 1)(U_1,u) = (U_1,u)$$
whilst using the last two pairs of the proposed skeleton we obtain:
$$(V_1, v) = (\emptyset, 1)(V_1,v), \qquad (a^{-1}A, 1)(V_1,v) =(a^{-1}A, 1)(V,v), \qquad  (\emptyset, 1)(V,v) = (V,v).$$
Thus it suffices to show that there exist elements $(T_i, t_i) \in \mathcal{S}(M)$ such that
$$(U_1, u) = (\emptyset, p_1)(T_1, t_1), (\emptyset, q_1)(T_1, t_1) = (\emptyset, p_2)(T_2,t_2), \ldots,  (\emptyset, q_n)(T_n, t_n) = (V_1,v).$$
Equivalently, we need subsets $T_i \subseteq M$ with 
$$U_1 = p_1T_1, q_1T_1 = p_2T_2, \ldots, q_nT_n = V_1$$
and elements $t_i \in M$ such that $u=p_1t_1, q_1t_1 = p_2t_2, \ldots, q_nt_n$. The latter is secured since $M$ is strongly finitely right equated with the given skeleton. We focus then on constructing the subsets $T_i$.

Consider 
$$B: = \{(h,k): h \in U_1, k \in V_1, ah=ak\} \subseteq \mathbf{r}(a).$$
For all $(h,k) \in B$ there exist $t_1(h,k), \ldots, t_n(h,k) \in M$ such that
$$h = p_1t_1(h,k), \, q_1t_1(h,k) = p_2t_2(h,k), \,\ldots, \,q_nt_n(h,k) = k.$$
Setting $T_i  =\{t_i(h,k): (h,k) \in B\}$ for $i=1, \ldots, n$ we claim that $U_1 = p_1T_1, q_1T_1 = p_2T_2, \ldots, q_nT_n=V_1$, as required to complete the proof. Since $a U_1 = aV_1$, for each $h \in U_1$ there exists at least one $k \in V_1$ with $ah=ak$, and so each element $h \in U_1$ appears as a possible first co-ordinate of $B$, giving $p_1T_1 = U_1$. Similarly, each element $k \in V_1$ appears as a possible second co-ordinate of $B$, giving $q_nT_n = V_1$. For $1 \leq i \leq n-1$ we also have $y \in q_iT_i$ if and only if there exists $(h,k) \in B$ such that $y=q_it_i(h,k)$ if and only if there exists $(h,k) \in B$ such that $y=p_{i+1}t_{i+1}(h,k)$ if and only if $y \in p_{i+1}T_{i+1}$.
\end{proof}

Putting Theorem \ref{thm:srih}  together with Theorem \ref{thm:sfre} yields the following:
\begin{corollary}\label{cor:rabund} 
Let $M$ be a strongly finitely right equated monoid. Then $\mathcal{S}(M)$ is weakly right coherent if and only if $M$ is strongly right ideal Howson.
\end{corollary}

Recalling from \cite{FK} that a monoid is left (right) LCM if it is right (left) cancellative and principally left (right) ideal Howson, we have: 
\begin{corollary}
\label{cor:LCM}
If the monoid $M$ is left (right) LCM, then $\mathcal{S}(M)$ is weakly left (right) coherent.
\end{corollary}

\begin{proof} The result  follows directly from Corollary \ref{cor:rcanc} and Corollary \ref{cor:rabund}.
\end{proof}

The class of left LCM monoids includes: groups, free monoids, the multiplicative
monoid of non-zero elements in any LCM domain, and Artin monoids (see \cite{FK} for details).

\medskip
Our next lemma will be used show that there exist finitely right equated monoids $M$ such that $\mathcal{S}(M)$ is not finitely right equated. 
\begin{lemma}
\label{lem:FREobstruct}
Let $M$ be an infinite monoid and suppose that there exists $x \in M$ with the following properties:
\begin{itemize}
\item if $p$ is right invertible in $M$ and $xp=xq$ for some $q \in M$, then $p=q$;
\item for any finite subset $A$ of $M$, there exists $u,v \in M \setminus A$ such that $xu=xv$, but $u \neq v$.
\end{itemize}
Then $\mathcal{S}(M)$ is not finitely right equated.
\end{lemma}
\begin{proof}
We show that $\mathbf{r}((\emptyset, x))$ is not finitely generated. Suppose for contradiction that $Y$ is a finite symmetric generating set for $\mathbf{r}((\emptyset, x))$. Let $A$ be the union of all sets $P$ such that $|P| = 1$ and there exists $((P, p), (Q,q)) \in Y$ for some $p,q \in M$ and some $Q \subseteq M$. Since $Y$ is finite, so too is $A$. Let $u,v \in M\setminus A$ with $xu=xv$ and $u \neq v$ and consider the elements $(\{u\},1), (\{v\},1) \in \mathcal{S}(M)$. We have
$$(\emptyset, x)(\{u\},1) = (\{xu\},x) = (\{xv\},x) = (\emptyset, x)(\{u\},1).$$
Thus, by assumption, there must be a finite $Y$-sequence from $(\{u\},1)$ to $(\{v\},1)$. Suppose this sequence is:
$$(\{u\},1) = (P_1,p_1)(T_1, t_1), (Q_1, q_1)(T_1, t_1) = (P_2, p_2)(T_2,t_2), \ldots, (Q_n,q_n)(T_n,t_n) = (\{v\},1),$$
where for $i=1, \ldots, n$ we have $((P_i,p_i),(Q_i,q_i)) \in Y$ and $(T_i,t_i) \in \mathcal{S}(M)$.  For all $i$ we have
$$xP_i =xQ_i \mbox{ and } xp_i = xq_i.$$
Looking at the second component of the $Y$-sequence, we also have the sequence:
$$1 = p_1t_1, q_1t_1 = p_2t_2, \ldots, q_nt_n = 1.$$
Since $p_1$ is right invertible and $xp_1=xq_1$, we deduce that $q_1=p_1$ and $1=p_2t_2$. By induction we then find that for $1\leq i\leq n$ each $p_i$ is right invertible, $p_i=q_i$  and $1=p_it_i$. We now show by induction  that for $1\leq i\leq n$ we have $P_i = Q_i = \emptyset$ and $p_iT_i = \{u\}$. Looking at the first component of the first equation in our sequence, we have that $\{u\} = P_1 \cup p_1T_1$, which forces $|P_1| \leq 1$. Since $u \not\in A$ we deduce that $P_1 = \emptyset$ and $p_1T_1 = \{u\}$. Since $xP_i=xQ_i$, this also forces $Q_i=\emptyset$ and the base case holds. Suppose then that $P_i = Q_i = \emptyset$ and $p_iT_i = \{u\}$ for some $i<n$. Then the equation
$$(Q_i,q_i)(T_i,t_i) = (P_{i+1},p_{i+1})(T_{i+1},t_{i+1})$$
yields
$$P_{i+1} \cup p_{i+1}T_{i+1} = Q_{i} \cup q_{i}T_{i} = \emptyset \cup p_{i}T_{i} = \{u\}.$$
Thus $|P_{i+1}| \leq 1$ and since $u \not\in A,$ we find that $P_{i+1} = \emptyset$ (which in turn forces $Q_{i+1} = \emptyset$) and $p_{i+1}T_{i+1} = \{u\}$. But then we have the contradiction
$$\{v\} = Q_n \cup q_nT_n = \emptyset \cup p_nT_n = \{u\}.$$\end{proof}
\begin{example} We now  use of a  construction of John Fountain, (see \cite{G}), adapted to our needs, to demonstrate that there is a finitely right equated monoid $M$ satisfying the conditions of the previous lemma. Let $G$ be a finitely generated infinite group and let $x,x^2,x^3$ be distinct symbols not contained in $G$. Consider the set of formal symbols $M = G \cup xG \cup x^2G \cup \{x^3,x^4\}$,  with product determined as follows (where we write $x^0=1$):
$$ab = \begin{cases}
x^{i+j}gh & \mbox{ if } a=x^ig, b=x^jh, \mbox{ with } 0 \leq i, j \leq 2 \mbox{ and }i+j \leq 2\\
x^3 & \mbox{ if } a=x^ig, b=x^jh, \mbox{ with } 1 \leq i, j \leq 2 \mbox{ and }i+j =3\\
x^3 & \mbox{ if } a=x^0g \in G, b=x^3 \mbox{ or } a=x^3, b=x^0h \in G\\
x^4&\mbox{ otherwise. } 
\end{cases}$$
First, we show that $M$ satisfies the conditions of Lemma \ref{lem:FREobstruct} with respect to the element $x$. Indeed, if $p$ is right invertible, then $p \in G$ and $xp=xq$ implies that $xq \in xG$ and moreover $p=q$, giving that the first condition holds. For the second condition, since $G$ is infinite, for any finite subset $A$ there exist distinct elements $u',v' \in G$ such that $u=x^2u', v=x^2v'$ are distinct elements in $  x^2G \setminus A$ and $xu = x^3=xv$.

It remains to show that $M$ is finitely right equated or, in other words, to show that each of the right annihilators of its elements is finitely generated. We denote the identity relation in $M$ by $\iota_M$ and the universal relation in $G$ by $\omega_G$. The fact that $G$ is finitely generated (as a group) by a (finite) subset $H$, where we can assume $H=H^{-1}$, gives that 
$\omega_G$ is finitely generated as a right congruence by $H\times H$.  Fix $g \in G$. Then $\mathbf{r}(g) = \iota_M$  is clearly finitely generated by $\emptyset$. Next, \begin{eqnarray*}
\mathbf{r}(xg) &=& \{(b,c): b,c \in M, (xg) b = (xg) c\} \\&=& \iota_M \cup \{(x^2u, x^2v): u, v \in G\} \cup \{ (x^3, x^4), (x^4, x^3) \}.
\end{eqnarray*} Since for any $u,v\in G$ there is a finite $H\times H$-sequence in $G$ from $u$ to $v$, it follows that there is a finite $x^2H\times x^2H$-sequence from $x^2u$ to $x^2v$ in $M$. Thus $\mathbf{r}(xg)$ is finitely generated by
$\{ (x^2h,x^2k), (x^3,x^4): h, k \in H\}.$ 
Similarly, \begin{eqnarray*}
\mathbf{r}(x^2g) &=& \{(b,c): b,c \in M, (x^2g) b = (x^2g) c\} \\&=& \iota_M\cup \{(xu, xv): u, v \in G\} \cup \{(b,c): b, c \in x^2 G\cup \{x^3,x^4\}\}
\end{eqnarray*} is finitely generated by
$\{ (xh,xk), (x^2,x^4): h, k \in H\}.$
We also have \begin{eqnarray*}
\mathbf{r}(x^3) &=& \{(b,c): b,c \in M,\, x^3 b = x^3 c\} \\&=& \{(b,c): b,c \in G\} \cup \{(b, c):  b, c \in xG \cup x^2 G\cup \{x^3,x^4\}\}
\end{eqnarray*} is finitely generated by
$\{(h,k),  (x,x^4): h, k \in H\}.$  Finally,
\begin{eqnarray*}
\mathbf{r}(x^4) &=& \{(b,c): b,c \in M\}
\end{eqnarray*} is finitely generated by
$\{(1,x^4)\}$.
\end{example}

In the proof of Lemma \ref{lem:FREobstruct} we argued that, under the conditions of the Lemma, there exists a right annihilator congruence of the form $\mathbf{r}((\emptyset, x))$ that is not finitely generated for some $x \in M$. In fact, to determine whether $\mathcal{S}(M)$ is is finitely right equated, it always suffices to consider right annihilator congruences of the form $\mathbf{r}((\emptyset, x))$ for $x \in M$.

\begin{proposition}
\label{prop:only}
Let $M$ be a monoid. The  expansion $\mathcal{S}(M)$ of $M$  is finitely right equated if and only if for all $a \in M$ the right annihilator congruence $\mathbf{r}((\emptyset, a))$ is finitely generated.
\end{proposition}

\begin{proof}
The forward direction is trivial. Suppose then that each right congruence $\mathbf{r}((\emptyset, a))$ is finitely generated. Fix $(A,a) \in \mathcal{S}(M)$
and let $X$ be a finite generating set for $\mathbf{r}((\emptyset, a))$. It is easy to see that $\mathbf{r}((\emptyset, a)) \subseteq \mathbf{r}((A, a))$: indeed, if $(\emptyset, a)(P,p) = (\emptyset, a)(Q,q)$, then $aP = aQ$ and $ap = aq$, from which it immediately follows that $A \cup aP = A \cup aQ$ and $ap = aq$, giving $(A, a)(P,p) = (A, a)(Q,q)$.

Now consider $((U,u), (V,v)) \in \mathbf{r}((A,a))$. Arguing as in the proof of Theorem \ref{thm:sfre}, setting
$U = U_1 \cup U_2$  and $V = V_1 \cup V_2$ where $U_1 = U \setminus a^{-1}A$, $U_2 = U \cap a^{-1}A$, $V_1 = V \setminus a^{-1}A$ and $V_2 = V \cap a^{-1}A$, 
we find that 
\begin{eqnarray*}
(U,u) &=& (\emptyset, 1)(U,u), \qquad  (a^{-1}A, 1)(U,u) = (a^{-1}A, 1)(U_1,u),  \qquad (\emptyset, 1)(U_1,u) = (U_1,u)\\
(V_1, v) &=& (\emptyset, 1)(V_1,v), \qquad (a^{-1}A, 1)(V_1,v) =(a^{-1}A, 1)(V,v), \qquad  (\emptyset, 1)(V,v) = (V,v),
\end{eqnarray*}
where $(A,a)(\emptyset, 1) = (A,a)(a^{-1}A,1)$.
Since there is an $X$-sequence from $U_1$ to $V_1$, 
setting $Y = X \cup \{((\emptyset, 1), (a^{-1}A, 1)), ((a^{-1}A, 1), (\emptyset, 1))\}$ it follows that we can construct a $Y$-sequence from $U$ to $V$. Since $Y$ is finite, this completes the proof.
\end{proof}

\begin{remark}
In light of Proposition \ref{prop:only} and the proof of Theorem \ref{thm:sfre}, to determine whether $\mathcal{S}(M)$ is finitely right equated it suffices to focus on the right congruences $\mathbf{r}((\emptyset, a))$ such that $a \in M$ is not $\mathcal{L}^*$-related to an idempotent.
\end{remark}

If $M$ is strongly finitely right equated, then we have shown in Theorem~\ref{thm:sfre} that $\mathcal{S}(M)$ is strongly finitely right equated, that is, for any fixed $(A,a)\in \mathcal{S}(M)$ there is a finite set generating $W$ for $\mathbf{r}((A,a))$ with respect to which any two elements of $\mathbf{r}((A,a))$ are connected by a $W$-sequence with {\em the same} skeleton.  Likewise, the proof of Theorem~\ref
{thm:FLE} shows that if $\mathcal{S}(M)$ is finitely left equated then 
for any $(A,a)\in \mathcal{S}(M)$ there is a 
finite generating set $Z$ for $\mathbf{l}((A,a))$ such that any pair of elements in $\mathbf{l}((A,a))$ may be connected via a $Z$-sequence having one of finitely many skeletons; this is equivalent to there being a bound on the length of $Z$-sequences required. Further, if $M$ is right cancellative, then we can take a single skeleton. 
\begin{question} \label{qn:skeletons} Characterise the monoids  $M$ such that $\mathcal{S}(M)$ is finitely right equated. Does this characterisation imply that there is a bound on the length of sequences required? \end{question}

\bigskip
The article focuses exclusively on the expansion $\mathcal{S}(M)$ of a monoid $M$. As indicated in the Introduction, this is the largest of a suite of possible expansions, where we make various restrictions on the pairs $(A,a)\in \mathcal{P}(M)\times M$  in the expansion. These include the restriction that $A$ is finite, resulting in the corresponding expansion $\mathcal{S}^f(M)$, and the restriction that $A$ is a finite set containing $\{1,a\}$ as a subset, resulting in the Szendrei expansion Sz$(M)$. In the latter case the expansion is to a monoid subsemigroup of $\mathcal{S}(M)$ with identity  $(\{ 1\},1)$. 

\begin{question}\label{qn:others} What conditions must a monoid $M$ satisfy such that $\mathcal{S}^f(M)$, Sz$(M)$, or other natural related expansions, have the finitary conditions investigated in this article? \end{question} 

We remark that the techniques we have developed here (which for example make frequent use of the fact that the sets $A$ can be chosen to be empty or infinite)  in some  cases would not apply to these `smaller' expansions so that new strategies would be required.


\begin{thebibliography}{99}

\bibitem{BR84} J.-C. Birget and J. Rhodes, Almost finite expansions of arbitrary semigroups, {\em J. Pure
Appl. Algebra} {\bf 32} (1984) 239-287.

\bibitem{BGR23} M. Brookes, V. Gould and N. Ru\v{s}kuc, Coherency properties for monoids of transformations and partitions, {\em Mathematika} {\bf 71} (2025) e70005.

\bibitem{CG21} S. Carson and V. Gould, Right ideal Howson semigroups, {\em Semigroup Forum} {\bf 102} (2021) 62-85.

\bibitem{DGM24} L.M. Dasar, V. Gould and C. Miller, Weakly right coherent semigroups, {\em Semigroup Forum} {\bf 
111} (2025) 392-414.

\bibitem{FG90} J. Fountain and  G.M.S Gomes,  The Szendrei expansion of a semigroup,
{\em Mathematika} {\bf  37} (1990) 251-260.

\bibitem{FGG99} J. Fountain, G.M.S Gomes and   V. Gould, Enlargements, semiabundancy and unipotent monoids, {\em Comm. Algebra} {\bf 27}  (1999)  595-614.

\bibitem{FK} J. Fountain and M. Kambites, Graph products of right cancellative monoids, \emph{J. Australian Math. Soc.} \textbf{87}  (2009) 227-252. 

\bibitem{Go871} V. Gould, 
Axiomatisability problems for $S$-systems, {\em 
J. London Math. Soc. (2)} {\bf 35} (1987) 193-201.

\bibitem{Go87} V. Gould, Model companions of S-systems, {\em  Quart. J. Math.} {\bf  38} (1987) 189-211.

\bibitem{G} V. Gould, Coherent monoids, \emph{J. Australian Math. Soc. Ser. A} \textbf{53} (1992) 166-182 .

\bibitem{GH} V. Gould and M. Hartmann, Coherency, free inverse monoids and related free algebras, \emph{Math. Proc. Cambridge Phil. Soc.} \textbf{163}  (2017) 23-45.

\bibitem{GHS} V. Gould, M. Hartmann and L. Shaheen, On some finitary conditions arising from the axiomatisability of certain classes of monoid acts, \emph{Comm. Algebra} \textbf{42}  (2014) 2584-2602.

\bibitem{GJ25}  V. Gould and M. Johnson, Forbidden configurations for coherency, \texttt{arXiv:2506.11321}.

\bibitem{Ho07} 
C. Hollings, 
Conditions for the prefix expansion
of a monoid to be (weakly) left ample, {\em Acta Sci. Math. (Szeged)} 
{\bf 73} (2007) 519-545.

\bibitem{H}  J.M. Howie, {\em Fundamentals of Semigroup Theory}, Oxford University Press (1995).

\bibitem{KKM} M. Kilp, U. Knauer and A.V. Mikhalev, {\em Monoids, Acts, and Categories}, de Gruyter, Berlin (2000).

\bibitem{Sh12} L. Shaheen,  Axiomatisability problems for S-acts, revisited {\em  Semigroup Forum} {\bf 85} (2012) 337-360.

\bibitem{Sz89}  M.B. Szendrei, A note on the Birget-Rhodes expansion of groups, {\em J. Pure Appl. Algebra}
{\bf 58} (1989), 93-99. 

\bibitem{Wh76} W. H. Wheeler, Model-companions and definability in existentially  complete structures, {\em Israel J. Math.} {\bf  25} (1976), 305-330.

\end{thebibliography}
\end{document}